\documentclass[reqno,11pt, psfig]{amsart}
\usepackage{amssymb, amsmath}
\usepackage{amsthm, amsfonts,mathrsfs}
\usepackage[dvips]{graphics}
\usepackage[T1]{fontenc}
\usepackage{times}
\usepackage{epsfig}
\usepackage{color}
\usepackage[all]{xypic}
\allowdisplaybreaks

\makeatletter \@addtoreset{equation}{section} \makeatother

\def\inte#1{
\displaystyle\mathop{#1\kern0pt}^\circ }

\let\S=\Sigma

\def\tilde{\widetilde}

\newcommand{\R}{\mathbb{R}}

\newtheorem{remark}{Remark}[section]
\newtheorem{lemma}{Lemma}[section]
\newtheorem{definition}{Definition}[section]

\newtheorem{theorem}{Theorem}[section]

\newcommand{\beq}{\begin{equation}}
\newcommand{\eeq}{\end{equation}}
\newcommand{\bes}{\begin{eqnarray}}
\newcommand{\ees}{\end{eqnarray}}

\def\com#1{\quad{\textrm{#1}}\quad}

\def\nn{\nonumber}

\def\LC{\left(}
\def\LB{\left[}
\newcommand{\RC}{\right)}
\def\RB{\right]}
\def\LV{\left|}
\def\RV{\right|}
\def\LN{\left\|}
\def\RN{\right\|}
\def\LCB{\left\{}
\def\RCB{\right\}}

\def\C{{\mathcal C}}
\def\L{{\bf L}}

\def\v{\vskip 1em}
\def\ve{\varepsilon}
\def\vp{\varphi}

\def\ds{\displaystyle}

\def\bega{\begin{array}}
\def\enda{\end{array}}
\def\begi{\begin{itemize}}
\def\endi{\end{itemize}}
\def\bel{\begin{equation}\label}

\begin{document}

\title[Global conservative solutions for the Novikov equation]{Existence and uniqueness of the global conservative weak solutions for the integrable Novikov equation}

\author[G. Chen]{Geng Chen}
\address{Geng Chen \newline
School of Mathematics, Georgia Institute of Technology, Atlanta, GA 30332}
\email{gchen73@math.gatech.edu}
\author[R.M. Chen]
{Robin Ming Chen}
\address{Robin Ming Chen\newline
Department of Mathematics\\
University of Pittsburgh\\
Pittsburgh, PA 15260}
\email{mingchen@pitt.edu}
\author[Y. Liu]
{Yue Liu}
\address{Yue Liu \newline
Department of Mathematics, University of Texas at Arlington, Arlington, TX 76019-0408}
\email{yliu@uta.edu}

\date{}
\maketitle

\begin{abstract}
The integrable Novikov equation can be regarded as one of the Camassa-Holm-type  equations with cubic nonlinearity. In this paper, we prove the global existence and uniqueness of the H\"older continuous energy conservative solutions for the Cauchy problem of the Novikov equation.
\end{abstract}

\noindent {\small 2010\textit{\ Mathematical Subject Classification:}  35A01, 35A02, 35D30, 35G20.}

\section{Introduction}
Consideration  here is the initial-value problem  for the Novikov equation in the form
\begin{equation}\label{e1.1}
 u_t-u_{xxt}+4u^2u_x=3uu_xu_{xx}+u^2u_{xxx},\quad x\in
 \mathbb{R},\quad t>0.\end{equation}
This equation was proposed by Novikov \cite{nov} in a symmetry
classification of nonlocal partial differential equations with cubic
nonlinearity. The Novikov equation \eqref{e1.1} is among the class of integrable equations with
the Lax pair given as \cite{nov}
\begin{equation*}
 \begin{split}
 \psi_{xxx} &=\psi_x +\lambda m^2\psi +2\frac{m_x}{m}\psi_{xx}
 +\frac{mm_{xx}-2m_x^2}{m^2}\psi_x,\\
 \psi_t &=\frac{u}{\lambda m}\psi_{xx} -\frac{mu_x+um_x}{m^2}\psi_x
 -u^2\psi_x,
 \end{split}
\end{equation*} where $ m = (1 - \partial_x^2) u. $
A $ 3 \times 3 $  matrix Lax pair
representation to the Novikov equation was derived by Hone and Wang  \cite{hw}.  Indeed, it can be shown
that the Novikov equation is related to a negative
flow in the Sawada-Kotera hierarchy. It is also found that
the Novikov equation admits a bi-Hamiltonian structure \cite {hw} 
\begin{equation*}
 m_t=\mathcal{J}_2\,\frac{\delta H_1}{\delta
 m}=\mathcal{J}_1\,\frac{\delta H_2}{\delta m},
\end{equation*}
with the Hamiltonian operators
\begin{equation*}
 \begin{split}
  \mathcal{J}_2 &=
  -2(3m\partial_x+2m_x)(4\partial_x-\partial_x^3)^{-1}(3m\partial_x+m_x),\\
  \mathcal{J}_1 &=
  (1-\partial_x^2)\frac{1}{m}\partial_x\frac{1}{m}(1-\partial_x^2),
 \end{split}
\end{equation*}
and the corresponding Hamiltonians
\begin{equation*}
 \begin{split}
  H_1 = \frac{1}{3}\,\int (m^{-8/3}m_x^2 +9m^{-2/3})\,dx, \quad {\rm and} \quad
  H_2 = \frac{1}{8}\,\int \left (u^4 +2u^2u_x^2-\frac{1}{3}u_x^4 \right )\,dx.
 \end{split}
\end{equation*}

The relevance of the Novikov equation \eqref{e1.1} to the Camassa-Holm (CH) equation \cite{ch,ff} can be revealed from the following compact form
\begin{equation}\label{b1.1}
 m_t+u^2m_x+\frac 32 (u^2)_xm =0, \quad m
 =u-u_{xx},
\end{equation}
whereas the CH  equation can be written as
\begin{equation}\label{e1.2}
m_t + u m_x + 2 u_x m = 0,\;\; m=u-u_{xx}.
\end{equation}


The classical CH equation \eqref{e1.2}
was originally derived as a model for surface waves, and has been studied extensively in the past two decades because of its many remarkable properties: infinity of conservation laws and complete integrability \cite{ch,ff}, existence of peaked solitons and multi-peakons \cite{ch, cht}, geometric formulations \cite{cq, mis}, well-posedness and  breaking waves, meaning solutions remain bounded while their slope becomes unbounded in finite time \cite{con1,ce-1,ce-2,ce-3,well}.  In particular, breaking waves are commonly observed in the ocean and are important for a variety of reasons, but surprisingly little is known about them. Indeed, breaking waves place large hydrodynamic loads on man-made structure, transfer horizontal momentum to surface currents, provide a source of turbulent energy to mix the upper layers of the ocean, move sediment in shallow water, and enhance the air-sea exchange of gases and particulate matter \cite {Co}. To further  understand why waves break and what happens during and after breaking themselves, we must first investigate the dynamics of wave breaking. Research work on breaking waves can be divided into three categories: those concerning waves (1) before, (2) during, and (3) after breaking.  Although up to now significant advances have been made in understanding the processes leading to the breaking, there are still some aspects of these questions unanswered, in particular, question (3), what happens after breaking of those waves.

Due to the formation of singularities of the strong solutions, it becomes imperative to consider weak solutions. One is now confronted with two major challenges. The first issue concerns the existence of weak solutions. In view of the hyperbolic structure of the equation, the existence theory is fairly robust. For example, in the context of the CH equation, Xin-Zhang \cite{XZ1,XZ2} obtain a class of global weak solutions via a vanishing viscosity approach. Such solutions are {\it diffusive} in nature. On the other hand, to incorporate the peakon-antipeakon interaction, Bressan-Constantin \cite{BC2,BC3} manage to construct two types of weak solutions to the CH equation, namely the {\it conservative} solution and the {\it dissipative} solution. Their approach is based on a nonlinear change-of-unknowns which allows them to transform the equation to a semi-linear system. The second difficulty is the non-uniqueness of weak solutions, leading to analytic and numerical barriers to the well-posedness. To resolve this issue, it is necessary to devise proper criteria for singling out admissible weak solutions. Usually, in the theory of continuum physics, such admissibility criteria can be induced through the Second Law of thermodynamics, in the form of the ``entropy" inequalities. In the case of the CH equation, it is shown by Bressan-Chen-Zhang \cite{BCZ}  that an energy criterion can be adapted as a selection principle for admissible weak solutions. Relying on the analysis of characteristics, they are able to convert the original initial-value problem to a set of ODEs satisfied by the solution $ u $ and $ u_x $ along the characteristics with a new energy variable. From the uniqueness of the solution to this ODEs system, they obtain the uniqueness of conservative solutions to the original equation. Such an approach is later extended to the case of the variational wave equations \cite{BCZ2}.
Another set of works \cite{BF,GHR1,GHR2} on the Lipschitz continuous dependence of solutions constructed in \cite{BC2} under some new metric is also of great interest.


Motivated by the works of Bressan-Constantin \cite{BC2}  and Bressan-Chen-Zhang \cite{BCZ,BCZ2}, we aim to investigate the issue on the existence and uniqueness of global weak solution to the Novikov equation.
We rewrite the initial value problem of the Novikov equation in the following weak form
\beq\label{nv}
\textstyle u_t +u^2 u_x +\partial_xP_1 + P_2=0,
\eeq
\bel{ID}
u(0,x)=u_0(x),
\eeq
where we denote that
\beq\label{P12}
\textstyle P_1:=p\ast (\frac{3}{2}u u_x^2+u^3),\qquad P_2:=\frac{1}{2} p\ast u_x^3,
\eeq
with
 $p=\frac{1}{2}e^{-|x|}$. The weak solution we seek here is defined in the following.
\begin{definition}\label{def1}
The energy conservative solution $u=u(t,x)$ of \eqref{nv}-\eqref{ID} satisfies
\begin{itemize}
\item[1.] For any fixed $t\geq 0$, $u(t,\cdot)\in H^1{(\R)}\cap W^{1,4}{(\R)}$.
The map $t\rightarrow u(t,\cdot)$ is Lipschitz continuous under the $L^4(\R)$ metric.

\item[2.] The solution $u=u(t,x)$ satisfies the initial condition (\ref{ID}) in $L^4(\R),$ and
\beq\label{nv_weak}
\bega{l}
\iint_{\Lambda}
	\left\{-u_x\, \bigl(\phi_t+u^2\,\phi_x)+(-\frac{3}{2}u u_x^2-u^3+P_1 +\partial_x P_2)\phi\right\}\, dx\, dt\\[4mm]
	\qquad\qquad+\int_{\R}u_{0,x}\phi(0,x)\,dx=0
	\enda
\eeq
for every test function $\phi\in C_c^1(\Lambda)$ with $\Lambda=\Big\{(t,x)\,\Big|\  t\in [0, \infty), x\in \R\Big\}$.

\item[3.] The solution $u=u(t,x)$ is {\bf\em conservative} if the balance law (\ref{ux4})
is satisfied in the following sense.

There exists a family of Radon measures $\{\mu_{(t)},\, t\in \R^+\}$,  depending continuously on time and w.r.t the topology of weak convergence of measures.
For every $t\in \R^+$,  the absolutely continuous part of $\mu_{(t)}$ w.r.t. the Lebesgue measure has density $u_x^4(t,\cdot)$, which provides a
measure-valued solution to the balance law
\bel{weak_en}
\bega{l}
 \int_{\R^+} \left\{ \int(\phi_t+u^2\phi_x)d\mu_{(t)}+\int\Big(4u^3 u_x^3-4u_x^3 (P_1+\partial_x P_2)\Big)\phi\, dx \right\}dt\\[4mm]
\qquad\qquad-\int_\R u^4_{0,x}\phi(0,x)dx=0,
\enda
\eeq
for every test function $\phi\in C_c^1(\Lambda)$.
\end{itemize}
\end{definition}

Now we state our main result on the global well-posedness of the energy conservative solutions to \eqref{nv}-\eqref{ID}.
\begin{theorem}\label{main}
Let $u_0\in H^1(\R)\cap W^{1,4}(\R)$ be an absolute continuous function on $x$. Then the initial value problem \eqref{nv}-\eqref{ID} admits a unique energy conservative solution $u(t,x) $ defined for all $(t,x)\in \R^+\times\R$. The solution also satisfies the following properties.
\begin{itemize}
\item[1.] $u(t,x)$ is H\"older continuous with exponent $3/4$ on both $t$ and $x$.
\item[2.] The first energy density $u^2+u_x^2$ is conserved for any time $t\geq 0$, i.e.
\bel{energy1}
\mathcal{E}(t)=\| u(t)\|^2_{H^1}=\| u_0\|^2_{H^1}\quad\hbox{for any}\quad t\geq0;
\eeq

\item[3.] The second energy density $u^4 + 2u^2u^2_x - {1\over3}u^4_x$ is conserved in the following sense.
\begin{itemize}
\item[(i).] An energy inequality is satisfied in $(t,x)$ coordinates:
\bel{energy2}
\mathcal{F}(t) =\int_\R \LC u^4 + 2u^2u^2_x - {1\over3}u^4_x \RC(t,x)\ dx \geq \mathcal{F}(0)\quad\hbox{for any}\quad t\geq0.
\eeq
\item[(ii).]  Denote  a family of Radon measures $\left\{\nu_{(t)},\  t\in \R^+\right\}$, such that
\[\nu_{(t)}(\mathcal{A})=\int_\mathcal{A} \LC u^4 + 2u^2u^2_x  \RC(t,x)\ dx-{1\over3}\mu_{(t)}(\mathcal{A})\]
for any  Lebesgue  measurable set $\mathcal{A}$ in $\R$. Then
for any $t\in \R^+$,
\[
\nu_{(t)}(\R)=\mathcal{F}(0)=\int_\R \LC u^4 + 2u^2u^2_x - {1\over3}u^4_x \RC(0,x)\ dx.
\]

For any $t\in \R^+$,  the absolutely continuous part of $\nu_{(t)}$ w.r.t. Lebesgue measure has density $u^4 + 2u^2u^2_x - {1\over3}u^4_x$.
For almost every $t\in \R^+$, the singular part of $\nu_{(t)}$ is concentrated on the set where
$u=0$.
\end{itemize}
\item[4.]
A continuous dependence result holds. Consider a sequence of initial data ${u_0}_n$ such that
$\|{u_0}_n-u_0\|_{H^1\cap W^{1,4}}\rightarrow 0$, as $n\rightarrow\infty$. Then
the corresponding solutions $u_n(t,x)$ converge to $u(t,x)$ uniformly for $(t,x)$
in any bounded sets.
\end{itemize}
\end{theorem}

Note that one of the differences between our result and the one for the CH equation lies in the concentration phenomenon of the energy measure, which is closely related to the persistence of the singularity. In fact from the blow-up criterion one can use an auxiliary variable $v = 2\arctan u_x$ to keep track of the wave breaking: wave breaking occurs exactly at $\cos v = -1$. It turns out that at the point of singularity the dynamics of $v$ follows $v_t \propto c'(u)$, where $c(u)$ is the wave speed. Therefore if $c'(u) \neq 0$ then $v$ would leave $-\pi$ immediately. This transversality condition implies that the singularity will disappear instantaneously. On the other hand, if at the place where $\cos v = -1$ one also has $c'(u) = 0$, then it is possible that this singularity will persist, leading to the concentration of energy. For the CH equation, the wave speed $c(u) = u$, which is non-degenerate. Therefore for almost all $t\ge 0$ the energy measure is absolutely continuous, which means that the wave breaking only happens instantaneously. However for the Novikov equation the wave speed $c(u)  = u^2$ exhibits degeneracy at places where $u = 0$, and the dynamics of $u$ involves the interplay between the local and nonlocal terms, cf. \eqref{semi}. It is thus possible that the solution remains zero for a period of time, and hence the higher order energy measure $\nu(t)$ (or $\mu(t)$) might have a singular part concentrated on the set where the solution vanishes. This corresponds to having a sustainable singularity.

Following the approach in \cite{BC2} and \cite {CS}, one of the key ideas is to construct a ``good" set of new variables based on the characteristics and to transform the original equation to a closed semi-linear system on these new variables. The choice for the new variables strongly hinges upon the structure of the {\it energy law} of the system. It turns out that one of the new variables, namely the energy dependent characteristic variable, can be constructed in a way such that it dilates the possible wave-breaking due to the concentration of characteristics, cf. \eqref{Y_def}. Moreover, the function spaces for the weak solutions are determined by the control of $u_x$. In the CH case the balance law for $u^2_x$ leads to a closed estimate for $\|u_x\|^2_{L^2}$, and hence it suffices to work in the $H^1$ regularity frame with the energy density being $1+u^2_x$. However for the Novikov equation, the balance law \eqref{ux2} for $u^2_x$ involves a higher order convolution $P_2$, defined in \eqref{P12}, which fails to be bounded by the $H^1$ energy. This suggests one to seek higher order balance laws (cf. \eqref{ux4}) and work in higher regularity spaces. It turns out that although such a balance law \eqref{ux4} including a term $-4u_x^3 \,\partial_x P_2$ still does not close the estimate for $\|u_x\|^4_{L^4}$ due to the strong {\it nonlocal} effects (this is in strong contrast to the case of variational waves as considered in \cite{CS}), it combines with some lower order balance laws to generate a conservation law $\mathcal{F}$, which is enough to provide controls on $\|u_x\|^4_{L^4}$. This motivates us to work in the $H^1\cap W^{1,4}$ space with the energy density chosen to be $(1+u^2_x)^2$.

With the set of new variables and some auxiliary unknown functions defined (cf. \eqref{def_lgq}) we are able to derive a closed semi-linear system on these new variables, cf. \eqref{semi}. The standard ODE theory asserts that the local well-posedness of the system is guaranteed provided that the right-hand side is Lipschitz. To extend the solution globally, some key  {\it a priori} estimates are achieved with the help of both the conservation laws on $\mathcal{E}$ and $\mathcal{F}$ defined in \eqref{energy1} and \eqref{energy2}, whereas in the case of the CH equation, only $\mathcal{E}$ is needed. The necessity of including the conservation law of $\mathcal{F}$, and hence enhancing the regularity requirement for the Novikov equation is mainly due to the higher nonlinearity in the equation. This enhanced regularity setup also changes the feature of the conservation of the lower order energy $\mathcal{E}(t) = \|u(t)\|^2_{H^1}$. In the CH equation, $\mathcal{E}$ is conserved for almost all time, whereas the better regularity of the Novikov solution leads to the exact conservation of $\mathcal{E}(t)$ for all time.

As for the uniqueness, we define a new energy variable as in \eqref{beta_def} which is in the spirit of \cite{BCZ}.
Such an energy variable helps to select the ``correct" characteristic after the collapse of characteristics at the time of wave breaking. This selection criterion is motivated by the idea of generalized characteristics used in \cite{Dafermos}. Furthermore, with some other auxiliary variables introduced associated to a given conservative solution, 
we can prove that these variables satisfy a particular semi-linear system which admits well-posedness, yielding the uniqueness of the conservative solution in the original variables. Again as we mentioned earlier, the wave speed of the Novikov equation has certain degeneracy, and hence the concentration of the energy is more delicate. We want to point that if energy concentration happens in an open set of time along any characteristic, one must have simultaneously that $u= 0$ and the auxiliary variable $v=2\arctan{u_x} = -\pi$.

The outline of the paper is as follows.  In Section 2 is a short review on conservation laws and basic  estimates for the Lipschitz continuous solution of semilinear system. A semi-linea system with a new variables is constructed in Section 3. Section 4 is devoted to global existence for the semi-linear system and the solutions are transferable to  those of the original equation \eqref{nv}. The crucial estimate and technical tool to determine a unique characteristic curve passing through every initial point is discussed in Section 5. Finally, according to various status of  the slope of conservative solution $ u_x,$ along a characteristic, a proof of the uniqueness of solution is  given in the last section, Section 6.

\section{Conservation laws and some estimates}
For smooth solutions, we differentiate \eqref{nv} with respect to $x$ to get
\beq\label{nv2}
\textstyle u_{xt} +u^2 u_{xx}+\frac{1}{2}u u_x^2
-u^3+P_1+ \partial_xP_2=0.
\eeq
Multiplying $u$ to \eqref{nv} we obtain
\beq\label{u2}
\LC {u^2\over2} \RC_t + \LC {1\over4} u^4 + u P_1 \RC_x + {1\over2}u P_2 = u_x P_1.
\eeq
Multiplying $u_x$ to \eqref{nv2}, we have
\[
\left(\frac{u_x^2}{2}\right)_t+u^2\left(\frac{u_x^2}{2}\right)_x +\frac{1}{2}u u_x^3-\left(\frac{u^4}{4}\right)_x+u_x P_1+u_x\partial_xP_2=0,
\]
which, upon using the identity $p_{xx} \ast f = p \ast f - f$, yields
\beq\label{ux2}
\left(\frac{u_x^2}{2}\right)_t + \LC {u^2u^2_x \over2} -{u^4\over4} + {1\over2}u \partial_x P_2 \RC_x - {1\over2}u P_2 = -u_x P_1.
\eeq
Therefore from \eqref{u2} and \eqref{ux2} we have the following local conservation law
\beq\label{cons1}
\LC {u^2 + u^2_x\over2} \RC_t + \LC {u^2u^2_x \over2} + uP_1 + u\partial_x P_2 \RC_x = 0.
\eeq

To derive another local conservation law, we multiply $4u^3$ to \eqref{nv} and $4u_x^3$ to \eqref{nv2} to
\begin{align}
&(u^4)_t + \LC {2\over3}u^6 + 4u^3P_1 \RC_x = -4u^3P_2 + 12 u^2u_x P_1, \label{u4}\\
&(u_x^4)_t+(u^2u_x^4)_x=4u^3 u_x^3-4u_x^3 (P_1+\partial_xP_2). \label{ux4}
\end{align}
From the above two, a direct computation shows that the following local conservation law holds.
\begin{equation}\label{cons2}
\begin{split}
&\LC u^4 + 2u^2u^2_x - {1\over3}u^4_x \RC_t + \LB 2u^4u^2_x - {1\over3} u^2u^4_x + {4\over3}u^3(P_1 + \partial_x P_2) \RB_x \\
& \quad +{4\over3} \LB (P_1 + \partial_x P_x)^2 - (P_2 + \partial_x P_1)^2 \RB_x = 0.
\end{split}
\end{equation}

Thus from \eqref{cons1} and \eqref{cons2} we have the following two conserved quantities
\begin{align}
&\mathcal{E}(t) : = \int_\R \LC u^2 + u^2_x \RC(t,x)\ dx = \mathcal{E}(0), \label{0cons E}\\
&\mathcal{F}(t) : = \int_\R \LC u^4 + 2u^2u^2_x - {1\over3}u^4_x \RC(t,x)\ dx = \mathcal{F}(0). \label{0cons F}
\end{align}
From the above two conservation laws and the Sobolev inequality
\[
\|u\|^2_{L^\infty} \leq \|u\|^2_{H^1} = \mathcal{E}(0)
\]
we deduce that
\beq\label{L^4 control}
\begin{split}
\|u_x\|^4_{L^4} & = 3\int_\R \LC u^4 + 2u^2u^2_x \RC\ dx - 3 \mathcal{F}(t) \\
& \leq 3\LC \|u\|^2_{L^\infty}\int_\R (u^2 + 2u^2_x)\ dx  - \mathcal{F}(t) \RC\\
& \leq 3\LC 2\mathcal{E}(t)^2 - \mathcal{F}(t) \RC = 3\LC 2\mathcal{E}(0)^2 - \mathcal{F}(0) \RC,
\end{split}
\eeq
which in turn implies that
\beq\label{L^3 control}
\|u_x\|^3_{L^3} \leq \sqrt{3\mathcal{E}(0) \LB 2\mathcal{E}(0)^2 - \mathcal{F}(0) \RB} =: K.
\eeq

From \eqref{L^3 control} we are able to bound $P_i(t)$ together with the derivatives $\partial_xP_i$ for $i = 1, 2$ as follows.
\beq\label{bound on P}
\begin{split}
\|P_1(t)\|_{L^\infty}, \|\partial_xP_1(t)\|_{L^\infty}& \leq \|p\|_{L^\infty}\left\| {3\over2}uu^2_x + u^3 \right\|_{L^1} \leq {3\over 4} \mathcal{E}(0)^{3/2},\\
\|P_1(t)\|_{L^2}, \|\partial_xP_1(t)\|_{L^2}& \leq \|p\|_{L^2}\left\| {3\over2}uu^2_x + u^3 \right\|_{L^1} \leq {3\over 2\sqrt{2}} \mathcal{E}(0)^{3/2},\\
\|P_2(t)\|_{L^\infty}, \|\partial_xP_2(t)\|_{L^\infty}& \leq {1\over2}\|p\|_{L^\infty}\|u^3_x\|_{L^1}\leq {1\over4}K,\\
\|P_2(t)\|_{L^2}, \|\partial_xP_2(t)\|_{L^2}& \leq {1\over2}\|p\|_{L^2}\|u^3_x\|_{L^1}\leq {1\over2\sqrt{2}}K.
\end{split}
\eeq

\section{Semi-linear system for smooth solutions}
Next we derive a semi-linear system for smooth solutions.

The equation of the characteristic is
\beq\label{0char}
\frac{d x(t)}{d t}=u^2\bigl(t,x(t)\bigr) .
\eeq
We denote the characteristic passing through the point $(t,x)$ as
\[
a\mapsto\, x^c(a;\, t,\, x)\,.
\]

As is explained in the Introduction, we use the energy density $(1+ u^2_x)^2$ and introduce new coordinates $(T,Y)$, such that
\beq\label{Y_def}
Y\equiv Y(t,x):=\int_0^{x^c(0;\, t,\, x)} (1+ u_x^2(0,x'))^2\,d x'.
\eeq
Here $Y=Y(t,x)$ is a characteristic coordinate, which satisfies
\beq\label{Y0}
Y_t+u^2 Y_x=0
\eeq
for any $(t,x)\in \mathbb{R}^+\times\mathbb{R}$. Denote
\[
T=t.
\]
Then any smooth function
$f(t,x)=f(T,x(T,Y))$ can be considered as a function of $(T,Y)$ also denoted by $f(T,Y)$. It is easy to check that
\[
f_t+u^2 f_x=f_Y(Y_t+u^2 Y_x)+f_T(T_t+u^2 T_x)=f_T.
\]
where we use \eqref{Y0}, $T_t=1$ and $T_x=0$.
On the other hand, we have
\[
f_x=f_Y Y_x +f_T T_x=f_Y Y_x.
\]
In a summary, we have
\beq\label{fyt}
f_T=f_t+u^2 f_x,\quad\hbox{and}\quad f_Y=\frac{f_x}{ Y_x}.
\eeq

We define
\beq\label{def_lgq}
v:=2\arctan{u_x} \com{and} \xi:=\frac{(1+u_x^2)^2}{Y_x}.
\eeq
Note that $Y_x$ measures the accumulation of characteristics, and thus $\xi$ can be thought of as the balance between the characteristic concentration and the energy density dilation.

Simple computation leads to
\beq\label{tri}
\frac{1}{1+u_x^2}=\cos^2 \frac{v}{2},\quad
\frac{u^2_x}{1+u_x^2}=\sin^2 \frac{v}{2} \com{and}
 \frac{u_x}{1+u_x^2}=\frac{1}{2}\sin v.
\eeq

In the next step, we derive a closed semi-linear system on unknowns $u$, $v$,
and $\xi$ under the independent variables $(T,Y)$. This method is first used by Bressan and Constantin in \cite{BC2} in establishing $H^1$ solution $u(t,\cdot)$
for Camassa-Holm equation modeling wave wave. In this paper, we use similar unknowns as those used in \cite{CS} which work for more general class of solutions.
First,
\beq\label{ueq2}
u_T=u_t +u^2 u_x=-\partial_xP_1-P_2
\eeq
with
\bes
\textstyle P_2(Y)&=&\frac{1}{4}\int_{-\infty}^{\infty} e^{-|x-\bar{x}|} u_x^3(t,\bar x)\,d\bar x\nn\\
&=&\frac{1}{8}\int_{-\infty}^{\infty}e^{-|\int^{Y}_{\bar Y}(\xi\,\cos^4{\frac{v}{2}})\,(T,\tilde Y) \,d\tilde Y|} (\xi\,\sin v\,\sin^2{\frac{v}{2}})\,(T,\bar Y)\,d\bar Y\nn
\ees
and similarly
\[
\textstyle \partial_xP_1(Y)=\frac{1}{2}\left(
\int_Y^\infty - \int_{-\infty}^Y\right)
e^{ -|
\int^Y_{\bar Y} (\xi\,\cos^4{\frac{v}{2}})\,(T,\tilde Y) \,d\tilde Y|}
 (\frac{3}{8}u\, \sin^2{v}+u^3\, \cos^4{\frac{v}{2}}) \,\xi\,
 (T,\bar Y)\,d\bar Y
\]
for any time $t=T$.

Next we calculate the equations for $v$ and $\xi$. First,
\bes\label{veq}
v_T&=&\frac{2}{1+u_x^2}\left(u_{xt}+u^2\, u_{xx} \right)\nn\\
		&=& \frac{2}{1+u_x^2}(-\frac{1}{2}u u_x^2
+u^3-P_1 - \partial_x P_2)\nn\\
		&=&-u \sin^2{\frac{v}{2}}+2u^3\,\cos^2{\frac{v}{2}}
		-2\cos^2{\frac{v}{2}}\, (P_1 + \partial_x P_2)\,.
\ees

To derive the equation for $\xi$, we need to use the following relation:
\beq\label{Y02}
Y_{tx}+u^2 Y_{xx}=-2u u_x Y_x.
\eeq which can be
derived from \eqref{Y0}. We have
\bes\label{xieq}
\xi_T&=&\left(\frac{(1+u_x^2)^2}{Y_x}\right)_T\\
&=&\frac{1}{Y_x}2(1+u_x^2)\left((u_x^2)_t+u^2(u_x^2)_x \right)-\frac{(1+u_x^2)^2}{(Y_x)^2}(Y_{xt}+u^2 Y_{xx})\nn\\
&=&\frac{1}{Y_x}\left((u_x^4)_t+u^2(u_x^4)_x + 2(u_x^2)_t+2u^2(u_x^2)_x\right)
+\frac{1+2u_x^2 +u_x^4}{Y_x} 2u u_x \nn\\
&=&\frac{1}{Y_x}\left((u_x^4)_t+(u^2u_x^4)_x+2(u_x^2)_t+2u^2(u_x^2)_x \right)+\frac{2u u_x +4u u_x^3}{Y_x}\nn\\
&=& \xi\,\frac{(4u^3+2u) u_x^3-(4u_x^3+4u_x) (P_1 + \partial_x P_2) +(4 u^3+2u)\,u_x }{(1+u_x^2)^2}\nn\\
&=&\xi\left[(2u^3+u)\, \sin^2\frac{v}{2} -2\, (P_1 + \partial_x P_2) +(2 u^3+u)\cos^2 \frac{v}{2} \right]\,\sin v\,,\nn\\
&=&\xi\left[ (2u^3 + u) - 2(P_1 + \partial_x P_2) \right] \sin v. \ \nn
\ees
%
\section{Global existence}
In the first subsection, we prove the global existence of solutions for
the semi-linear system derived in the previous section. Then we transform the solution for the semi-linear system to the original equation \eqref{nv}.

Note, in the first subsection, we start from the semi-linear system and do not use any more information in the previous section when we derive it, including the definitions of $v$ and $\xi$.

\subsection{Existence of semi-linear system}
The semi-linear system is
\begin{equation}\label{semi}\left\{\begin{array}{rcl}
u_T&=&-\partial_xP_1-P_2\\[2mm]
\displaystyle v_T&=&-u \sin^2{\frac{v}{2}}+2u^3\,\cos^2{{v\over2}}
		-2\cos^2{\frac{v}{2}}\, (P_1+\partial_xP_2)\\[2mm]
\displaystyle \xi_T&=&\xi\left[(2u^3+u) -2\, (P_1+\partial_xP_2)  \right]\,\sin v,
\end{array}\right.
\end{equation}
with initial conditions given as
\beq\label{initial}\left\{\begin{array}{rcl}
u(0,Y)&=&u_0(x(0,Y)),\\[2mm]
v(0,Y)&=&2\arctan(u'_0(x(0,Y))),\\[2mm]
\xi(0,Y)&=&1.
\end{array}\right.
\eeq
where
\begin{equation}\label{def P}
\begin{split}
P_1(Y)= & \frac{1}{2}\int_{-\infty}^{\infty}
e^{ -\left|
\int^Y_{\bar Y} (\xi\,\cos^4{\frac{v}{2}})\,(T,\tilde Y) \,d\tilde Y\right|}
 \left(\frac{3}{8}u\, \sin^2{v}+u^3\, \cos^4{\frac{v}{2}}\right) \,\xi\,
 (T,\bar Y)\,d\bar Y,\\
\partial_xP_1(Y)= & \frac{1}{2}\left(
\int_Y^\infty - \int_{-\infty}^Y\right)
e^{ -\left|
\int^Y_{\bar Y} (\xi\,\cos^4{\frac{v}{2}})\,(T,\tilde Y) \,d\tilde Y\right|}
 \left(\frac{3}{8}u\, \sin^2{v}+u^3\, \cos^4{\frac{v}{2}}\right) \,\xi\,
 (T,\bar Y)\,d\bar Y,\\
P_2(Y)= & \frac{1}{8}\int_{-\infty}^{\infty}e^{-\left|\int^{Y}_{\bar Y}(\xi\,\cos^4{\frac{v}{2}})\,(T,\tilde Y) \,d\tilde Y\right|} \left(\xi\,\sin v\,\sin^2{\frac{v}{2}}\right)\,(T,\bar Y)\,d\bar Y,\\
%
%
%
%
\partial_xP_2(Y)= &
\frac{1}{8}\left(
\int_Y^\infty - \int_{-\infty}^Y\right)e^{-\left|\int^{Y}_{\bar Y}(\xi\,\cos^4{\frac{v}{2}})\,(T,\tilde Y) \,d\tilde Y\right|} \left(\xi\,\sin v\,\sin^2{\frac{v}{2}}\right)\,(T,\bar Y)\,d\bar Y.
\end{split}
\eeq

It is easy to see that semi-linear system \eqref{semi} is invariant under translation by $2\pi$ in $v$.
It would be more precise to use $e^{iv}$ as variable. For simplicity, we use $v\in[-\pi,\pi]$ with endpoints identified.

Now we consider \eqref{semi} as a system of ordinary differential equations on $(u,v,\xi)$ in the Banach space
\beq\label{space}
X := \LB H^1(\R)\cap W^{1,4}(\R) \RB \times \LB L^2(\R)\cap L^\infty(\R) \RB \times L^\infty(\R)
\eeq
with the norm
\[
\|(u, v, \xi) \|_X := \|u\|_{H^1} + \|u\|_{W^{1,4}} + \|v\|_{L^2} + \|v\|_{L^\infty} + \|\xi\|_{L^\infty}.
\]

From the standard ODE theory it follows that to obtain the local wellposedness of solutions to the system \eqref{semi}-\eqref{initial}, it suffices to prove that all functions on the right-hand side of \eqref{semi} are locally Lipschitz continuous. Then the energy conservation would ensure the global wellposedness.

\begin{lemma}[Local wellposedness]\label{lem local exist}
Let $u_0\in H^1(\R)\cap W^{1,4}(\R)$. Then there exists a $T_0>0$ such that the initial value problem \eqref{semi}-\eqref{initial} has a unique solution defined on $[0,T_0]$.
\end{lemma}
\begin{proof}
The idea of the proof is similar to that from \cite{BC2}. For completeness we provide the details. As explained in the previous paragraph, our goal is to show that the right-hand side of \eqref{semi} is Lipschitz continuous in $(u, v, \xi)$ on every bounded domain $\Omega$ in $X$ of the form
\begin{align*}
\Omega = & \left\{ (u,v,\xi): \  \|u\|_{H^1}+\|u\|_{W^{1,4}}\leq A,\ \|v\|_{L^2}\leq B,\ \|v\|_{L^\infty} \leq {3\pi\over2}, \right.\\
& \qquad \xi^-\leq \xi(x) \leq \xi^+ \text{ a.e. } x \Big\},
\end{align*}
for any positive constants $\alpha, \beta, \xi^-, \xi^+$.

From the Sobolev inequality
\[
\|u\|_{L^\infty} \leq \|u\|_{H^1}, \quad \text{and } \quad \|u\|_{L^\infty} \leq C\|u\|_{W^{1,4}},
\]
and the uniform bounds on $v, \xi$ it follows that the maps
\[
-u\sin^2{v\over2}, \quad 2u^3\cos^2{v\over2}, \quad (2u^3 + u) \xi\sin{v}
\]
are all Lipschitz continuous from $\Omega$ into $L^2(\R)\cap L^4(\R)\cap L^\infty(\R)$. Thus what remains to prove is that the maps
\beq\label{P map}
(u, v, \xi) \mapsto (P_i, \partial_x P_i), \qquad i = 1, 2
\eeq
are Lipschitz from $\Omega$ into $L^2(\R)\cap L^4(\R)\cap L^\infty(\R)$. In fact we can show that the above maps are Lipschitz from $\Omega$ into $H^1(\R)\cap W^{1,4}(\R)$.

A crucial idea is to utilize the exponential decay in $P_i$ and $\partial_x P_i$ as $|Y-\bar{Y}|\to\infty$. For that purpose, we first notice that for $\displaystyle (u, v, \xi)\in \Omega$ it holds
\begin{align*}
\text{measure}&\left\{ Y\in\R:\ \left| {v(Y)\over2} \right| \geq {\pi\over4} \right\} \leq \text{measure}\left\{ Y\in\R:\ \sin^2{v(Y)\over2} \geq {1\over2} \right\}\\
&\leq 2\int_{\left\{ Y\in\R:\ \sin^2{v(Y)\over2} \geq {1\over2} \right\}} \sin^2{v(Y)\over2}\ dY \\
&\leq {1\over2}\int_{\left\{ Y\in\R:\ \sin^2{v(Y)\over2} \geq {1\over2} \right\}} v^2(Y)\ dY \leq {B^2\over2}.
\end{align*}
Thus for any $\bar{Y}< Y$ we have
\beq\label{exp kernel}
\begin{split}
\int^Y_{\bar{Y}} \xi(\tilde{Y})\cos^4{v(\tilde{Y})\over2}\ d\tilde{Y} & \geq \int_{\left\{ \tilde{Y}\in [\bar{Y}, Y],\ \left| {v(\tilde{Y})\over2} \right| \leq {\pi\over4} \right\}} {\xi^-\over4}\ d\tilde{Y} \\
& \geq {\xi^-\over4}\left[ (Y - \bar{Y}) -{B^2\over 2} \right],
\end{split}
\eeq
and we can define the following exponentially decaying function
\beq\label{exp decay func}
\Gamma(\zeta) := \min\left\{ 1, \exp\left( {B^2\xi^-\over8} - {|\zeta|\xi^-\over4} \right) \right\}.
\eeq
A direct computation shows that
\beq\label{gamma int}
\|\Gamma\|_{L^1} = \left( \int_{|\zeta|\leq {B^2\over2}} + \int_{|\zeta|\geq {B^2\over2}} \right) \Gamma(\zeta)\ d\zeta =B^2 + {8\over \xi^-}.
\eeq

We first show that $P_i, \partial_x P_i \in H^1(\R)\cap W^{1,4}(\R)$, that is,
\beq\label{P and P_x}
P_i, \quad \partial_Y P_i, \quad \partial_x P_i, \quad \partial_Y(\partial_x P_i) \in L^2(\R)\cap L^4(\R).
\eeq
It suffices to obtain the {\it a priori} estimates. For simplicity we will only consider $\partial_x P_i$. The estimates for $P_i$ are entirely similar.

From the definition \eqref{def P} we have
\begin{equation*}
\begin{split}
|\partial_x P_1(Y)| & \leq {\xi^+\over2} \left| \Gamma \ast \LB |u|\LC {3\over8}\sin^2v + u^2\cos^4{v\over2} \RC \RB(Y) \right|,\\
|\partial_x P_2(Y)| & \leq {\xi^+\over8} \left| \Gamma \ast \left( |\sin v| \sin^2{v\over2} \right)(Y) \right|.
\end{split}
\end{equation*}
Therefore by Young's inequality it follows that for $p = 2, 4$,
\begin{align*}
\|\partial_xP_1\|_{L^p} & \leq {\xi^+\over2} \|\Gamma\|_{L^1} \left( {3\over8} \|u\|_{L^p} + \|u\|_{L^\infty}\|u^2\|_{L^p} \right) < \infty, \\
\|\partial_xP_1\|_{L^p} & \leq {\xi^+\over8} \|\Gamma\|_{L^1} \cdot {1\over4}\|v^2\|_{L^p} \leq {\xi^+\over32} \|\Gamma\|_{L^1}\|v\|^{(2p-2)/p}_{L^\infty} \|v\|^{2/p}_{L^2} < \infty.
\end{align*}

Next differentiating $\partial_xP_i$ we have
\begin{align}
\partial_Y(\partial_x P_1)(Y) = & -\left[ {3\over8}u(Y)\sin^2v(Y) + u^3(Y)\cos^4{v(Y)\over2} \right] \xi(Y) \nn\\
& + \frac{1}{2}\left(
\int_Y^\infty - \int_{-\infty}^Y\right)
e^{ -\left|
\int^Y_{\bar Y} (\xi\,\cos^4{\frac{v}{2}})\,(\tilde Y,T) \,d\tilde Y\right|}\nn\\
& \cdot \xi(Y) \cos^4{v(Y)\over2}\text{sgn}(\bar{Y} - Y) \label{pYpxP1}\\
& \cdot \left(\frac{3}{8}u(\bar Y)\, \sin^2{v(\bar Y)}+u^3(\bar Y)\, \cos^4{\frac{v(\bar Y)}{2}}\right) \,\xi\,
 (\bar Y)\,d\bar Y,\nn\\
\partial_Y(\partial_x P_2)(Y) = & -{1\over4}\sin v(Y) \sin^2{v(Y)\over2}\xi(Y) \nn\\
& + \frac{1}{8}\left(
\int_Y^\infty - \int_{-\infty}^Y\right)
e^{ -\left|
\int^Y_{\bar Y} (\xi\,\cos^4{\frac{v}{2}})\,(\tilde Y,T) \,d\tilde Y\right|}\label{pYpxP2}\\
& \cdot \xi(Y) \cos^4{v(Y)\over2}\text{sgn}(\bar{Y} - Y)  \cdot \sin v(\bar Y) \sin^2{v(\bar Y)\over2}\xi(\bar Y) \ d\bar Y.\nn
\end{align}
Therefore
\begin{align*}
|\partial_Y(\partial_x P_1)(Y)| & \leq \xi^+\left( {3\over8}|u| + |u^3| \right) + {(\xi^+)^2\over2} \left| \Gamma \ast \LB |u|\LC {3\over8}\sin^2v + u^2\cos^4{v\over2} \RC \RB \right|, \\
|\partial_Y(\partial_x P_2)(Y)| & \leq {\xi^+\over4}|\sin v| + {(\xi^+)^2\over8} \left| \Gamma \ast \left( |\sin v| \sin^2{v\over2} \right) \right|.
\end{align*}
From previous estimates it suffices to bound $\xi^+\left( {3\over8}|u| + |u^3| \right)$ and ${\xi^+\over4}|\sin v|$ in $L^p$ for $p = 2, 4$. Indeed,
\begin{align*}
\left\| \xi^+\left( {3\over8}|u| + |u^3| \right) \right\|_{L^p} & \leq \xi^+\left( {3\over8}\|u\|_{L^p} + \|u\|^2_{L^\infty} \|u\|_{L^p} \right) <\infty,\\
\left\| {\xi^+\over4}|\sin v| \right\|_{L^p} & \leq {\xi^+\over4} \|v\|_{L^p} < \infty.
\end{align*}
The similar estimates can be obtained for $P_i, \partial_Y P_i$, and hence we prove \eqref{P and P_x}. Thus the maps defined in \eqref{P map} take values in $H^1(\R)\cap W^{1,4}(\R)$.

Next we check the Lipschitz continuity of the map given in \eqref{P map}. This can be proved by showing that for $(u, v, \xi)\in \Omega$, the partial derivatives
\beq\label{partial deriv}
\partial_u P_i, \ \partial_v P_i, \  \partial_\xi P_i, \ \partial_u (\partial_xP_i), \ \partial_v (\partial_xP_i), \  \partial_\xi (\partial_xP_i)
\eeq
are uniformly bounded linear operators from the appropriate spaces into $H^1(\R)\cap W^{1,4(\R)}$. Again we only detail the argument for $\partial_u (\partial_xP_1)$. All other derivatives can be estimated by the same methods.

For a given $(u, v, \xi)\in \Omega$ and for a test function $\phi \in H^1(\R)$, the operators $\partial_u (\partial_xP_1)$ and $\partial_u(\partial_Y\partial_xP_1)$ are defined by
\begin{align*}
\left[ \partial_u (\partial_xP_1)\cdot \phi \right](Y) = & \frac{1}{2}\left(
\int_Y^\infty - \int_{-\infty}^Y\right) e^{ -\left|
\int^Y_{\bar Y} (\xi\,\cos^4{\frac{v}{2}})\,(\tilde Y,T) \,d\tilde Y\right|} \\
& \cdot \left[ {3\over8} \sin^2v(\bar Y) + 3u^2(\bar Y) \cos^4{v(\bar Y)\over2} \right] \xi(\bar Y) \phi(\bar Y)\ d\bar Y,\\
\left[ \partial_u(\partial_Y\partial_xP_1) \cdot \phi \right] = & -\left[ {3\over8}\sin^2v(Y) + 3u^2(Y)\cos^4{v(Y)\over2} \right] \xi(Y)\cdot \phi(Y) \\
& + \frac{1}{2}\left(
\int_Y^\infty - \int_{-\infty}^Y\right)
e^{ -\left|
\int^Y_{\bar Y} (\xi\,\cos^4{\frac{v}{2}})\,(\tilde Y,T) \,d\tilde Y\right|}\\
& \cdot \xi(Y) \cos^4{v(Y)\over4}\text{sgn}(\bar{Y} - Y) \\
& \cdot \left(\frac{3}{8}\sin^2{v(\bar Y)}+3u^2(\bar Y)\, \cos^4{\frac{v(\bar Y)}{2}}\right) \,\xi\,
 (\bar Y) \phi(\bar Y)\,d\bar Y.
\end{align*}
Therefore for $p = 2, 4$,
\begin{align*}
\left\| \partial_u (\partial_xP_1)\cdot \phi \right\|_{L^p} & \leq {\xi^+\over2} \left\| \Gamma \ast \left( {3\over8} \sin^2v + 3u^2 \cos^4{v\over2} \right)  \right\|_{L^p} \|\phi\|_{L^\infty}\\
& \leq {\xi^+\over2} \|\Gamma\|_{L^1} \left( {3\over8}\|v\|_{L^p} + 3\|u^2\|_{L^p} \right) \|\phi\|_{H^1},\\
\left\| \partial_u(\partial_Y\partial_xP_1) \cdot \phi \right\|_{L^p} & \leq \xi^+ \left\|  {3\over8} \sin^2v + 3u^2 \cos^4{v\over2}  \right\|_{L^p} \|\phi\|_{L^\infty} \\
& \quad + {(\xi^+)^2\over2} \left\| \Gamma \ast \left( {3\over8} \sin^2v + 3u^2 \cos^4{v\over2} \right)  \right\|_{L^p} \|\phi\|_{L^\infty}\\
& \leq {\xi^+} \left(1+ {\xi^+\over2} \|\Gamma\|_{L^1} \right) \left( {3\over8}\|v\|_{L^p} + 3\|u^2\|_{L^p} \right) \|\phi\|_{H^1}.
\end{align*}
From the above two estimates we know that $\partial_u\partial_x P_1$ is a bounded linear operator from $H^1(\R)\cap W^{1,4}(\R)$ into $H^1(\R)\cap W^{1,4}(\R)$. The boundedness of other partial derivatives in \eqref{partial deriv} can be obtained in a similar way. Thus we have established the Lipschitz continuity of the maps \eqref{P map}.

Now we can apply the standard ODE theory in Banach spaces to prove that the Cauchy problem \eqref{semi}-\eqref{initial} admits a unique solution on $[0, T]$ for some $T>0$, establishing the local wellposedness.
\end{proof}

The next step is to extend the solution obtained in Lemma \ref{lem local exist} globally.
\begin{theorem}[Global wellposedness]\label{thm global exist}
Let $u_0\in H^1(\R)\cap W^{1,4}(\R)$. Then the initial value problem \eqref{semi}-\eqref{initial} admits a unique solution defined for all $T\geq 0$.
\end{theorem}
\begin{proof}
From the previous lemma we know that in order to continue the local solution, one needs to show that for all $T<\infty$,
\beq\label{global control}
\|u\|_{H^1} + \|u\|_{W^{1,4}} + \|v\|_{L^2} + \|v\|_{L^\infty} + \|\xi\|_{L^\infty} + \left\|{1\over \xi}\right\|_{L^\infty} < \infty.
\eeq

First we claim that
\beq\label{ueq}
u_Y=\frac{u_x}{Y_x}=\frac{u_x}{1+u_x^2}\,\frac{1}{1+u_x^2}\,\xi=\frac{1}{2}\,\xi\,\sin{v}\, \cos^2{\frac{v}{2}}.
\eeq
To get \eqref{ueq}, one needs to check
\[
u_{YT}=u_{TY}.
\]
From \eqref{semi} and \eqref{def P} we have
\bes
u_{YT}\label{uyt}
&=&\left(\frac{1}{2}\,\xi\,\sin{v}\, \cos^2{\frac{v}{2}}\right)_T\nn\\
	&=&\textstyle \frac{1}{2}\,\xi\,\sin^2{v}\, \cos^2{\frac{v}{2}}\,\Big\{(2u^3+u)\,
	\sin^2\frac{v}{2} -2\, (P_1+\partial_xP_2) +(2 u^3+u)\cos^2 \frac{v}{2} \Big\}\nn\\[2mm]
		&&\textstyle+\frac{1}{2}\,\xi\,(1-4\sin^2{\frac{v}{2}})\,
		\Big(-u \sin^2{\frac{v}{2}}+2u^3\,\cos^2{\frac{v}{2}}
		-2\cos^2{\frac{v}{2}}\, (P_1+\partial_xP_2)\Big)\, \cos^2{\frac{v}{2}}
		\nn\\
		&=&\textstyle \xi\left\{ \frac{3}{8}u\, \sin^2{v}+u^3\, \cos^4{\frac{v}{2}}-\cos^4{\frac{v}{2}}(P_1+\partial_xP_2)\right\}\nn\\
		&=&u_{TY}, \label{equality}
\ees
Moreover the initial conditions imply that at $T = 0$
\[
u_Y = \frac{1}{2}\,\sin{v}\, \cos^2{\frac{v}{2}}, \quad \text{and }\quad \xi = 1,
\]
which means that \eqref{ueq} holds initially. Therefore \eqref{equality} indicates that \eqref{ueq} holds for all $T$, as long as the solution exists.

Next we check the conservation laws \eqref{0cons E} and \eqref{0cons F}. In the new system, the conservation of ${E}(T)$ and ${F}(T)$ read
\begin{align}\label{cons E new}
&{E}(T) := \int_\R \LC u^2\cos^2{v\over2} + \sin^2{v\over2} \RC \xi \cos^2{v\over2} \ dY = {E}(0), \\
&{F}(T) := \int_\R \LC u^4 \cos^4{v\over2} + 2u^2 \cos^2{v\over2} \sin^2{v\over2} - {1\over3} \sin^4{v\over2} \RC \xi \ dY = {F}(0). \label{cons F new}
\end{align}

Some other identities about the $Y$-derivatives which will be used are the following.
\beq\label{other Y deriv}
\begin{split}
\partial_Y P_i & = \xi \cos^4{v\over2} \partial_x P_i, \quad i = 1,2,\\
\partial_Y(\partial_x P_1) & = - \LC {3\over8} u \sin^2 v + u^3 \cos^4 {v\over2} \RC\xi + \xi \cos^4{v\over2} P_1,\\
\partial_Y(\partial_x P_2) & = - {1\over4} \sin v \cos^2{v\over2} \xi+ \xi \cos^4{v\over2} P_2.
\end{split}
\eeq

Using \eqref{semi} we can write
\begin{equation*}
\begin{split}
{d{E}\over dT} = & \int_\R \LB \LC u^2\cos^4{v\over2} + {1\over4}\sin^2 v \RC \xi \RB_T\ dY \\
= & \int_\R \xi \left\{ -2u\cos^4{v\over2} (\partial_x P_1 + P_2) - \LC 2u^2 \cos^3{v\over2}\sin{v\over2} - {1\over2}\sin v\cos v \RC \right. \\
& \quad\quad \cdot \LB -u\sin^2{v\over2} + 2u^3\cos^2{v\over2} - 2\cos^2{v\over2}(P_1 + \partial_x P_2) \RB \\
& \quad\quad \left.+ \sin v\LC u^2\cos^4{v\over2} + {1\over4}\sin^2v \RC \LB (2u^3+u) - 2(P_1 + \partial_x P_2) \RB\right\} \ dY.
\end{split}
\end{equation*}
Then applying \eqref{ueq} and \eqref{other Y deriv} we end up with
\[
{d{E}\over dT} = \int_\R \LB u^4 - 2u(P_1 + \partial_x P_2) \RB_Y \ dY = 0,
\]
where in deriving the last equality we have used the asymptotic property $\lim_{|Y|\to\infty} u(Y) = 0$ as $u\in H^1(\R)$, and the fact that $P_i, \partial_x P_i \in H^1(\R)$ and hence uniformly bounded.

Similarly for the conservation of $\mathcal{F}$, after a long calculation we obtain
\begin{align*}
{d{F} \over dT} & = \int_\R \left\{ u^6 +{4\over3}\LB u^3(P_1 + \partial_x P_2) + (P_2 + \partial_xP_1)^2 - (P_1 + \partial_xP_2)^2\RB \right\}_Y\ dY\\
& = 0,
\end{align*}
where the last equality follows from the fact that $u, P_i, \partial_x P_i \in H^1(\R)$.

Now we have proved the conservation laws \eqref{cons E new} and \eqref{cons F new} in the new variables
along any solution of \eqref{semi}-\eqref{initial}. Hence a uniform {\it a priori} estimate on $\|u(T)\|_{L^\infty}$ can be obtained as
\beq\label{control u}
\begin{split}
\sup_{Y}|u^2(T, Y)| & \leq 2\int_\R |uu_Y|\ dY \leq \int_\R |u| \left|\sin v \cos^2{v\over2} \right|\xi\ dY \\
& \leq \int_\R 2|u| \left|\sin {v\over2} \cos^3{v\over2} \right|\xi\ dY \leq {E}(0).
\end{split}
\eeq

Next we want to use the equation for $\xi$ in \eqref{semi} to derive the $L^\infty$ bound for $\xi$. From definition \eqref{def P}  and \eqref{cons E new} we have
\beq\label{est P1}
\|P_1(T)\|_{L^\infty} \leq {1\over2} \left\| |u| \LC {3\over2} \sin^2{v\over2}\cos^2{v\over2} + u^2\cos^4{v\over2} \RC\xi \right\|_{L^1} \leq {3\over4} {E}(0)^{3/2}.
\eeq
For $\partial_x P_2$, we have
\[
\|\partial_xP_2(T)\|_{L^\infty} \leq {1\over8} \LC \int_\R \xi \sin^2v  \ dY \RC^{1/2} \LC \int_\R \xi \sin^4{v\over2}  \ dY \RC^{1/2}
\]
From definition \eqref{cons E new} we know
\[
\int_\R \xi \sin^2v  \ dY = \int_\R 4\xi \sin^2{v\over2} \cos^2{v\over2}  \ dY \leq 4 {E}(0).
\]
Further using \eqref{cons F new} and \eqref{control u} it follows that
\begin{align*}
\int_\R \xi \sin^4{v\over2}  \ dY & = 3 \LB \int_\R \LC u^4 \cos^4{v\over2} + 2u^2 \cos^2{v\over2} \sin^2{v\over2} \RC \xi \ dY - {F}(0) \RB\\
& \leq 3 \LB 2{E}(0)^2 - {F}(0) \RB.
\end{align*}
Putting together the above two estimates we arrive at
\beq\label{est pxP2}
\|\partial_xP_2(T)\|_{L^\infty} \leq {1\over4} \sqrt{3{E}(0) \LB 2{E}(0)^2 - {F}(0) \RB} = {1\over4}K,
\eeq
where $K$ is defined in \eqref{L^3 control}. This way we have in fact recovered the estimates \eqref{bound on P} in the new variables.

With the estimates \eqref{est P1}-\eqref{est pxP2}, it is now clear from the third equation in \eqref{semi} that
\[
|\xi_T| \leq \LB 2{E}(0)^{3/2} + {E}(0)^{1/2} + 2\LC {3\over4} {E}(0)^{3/2} + {1\over4}K \RC \RB \xi =: A_0 \xi.
\]
Together with the initial condition $\xi(0, Y) = 1$ we know that
\beq\label{control xi}
e^{-A_0T} \leq \xi(T) \leq e^{A_0T}.
\eeq

Similarly when plugging the estimates \eqref{control u} and \eqref{est P1}-\eqref{est pxP2} into the second equation of \eqref{semi} we find
\[
|v_T| \leq 2{E}(0)^{3/2} + {E}(0)^{1/2} + 2\LC {3\over4} {E}(0)^{3/2} + {1\over4}K \RC = A_0.
\]
Hence
\beq\label{control v}
\|v(T)\|_{L^\infty} \leq \|v(0)\|_{L^\infty} + A_0T.
\eeq

To obtain the estimates on $\|u(T)\|_{L^2}$ we multiply $u$ to the first equation of \eqref{semi} to deduce
\[
{d\over dT}\LC {1\over2}\|u(T)\|^2_{L^2} \RC \leq \|u(T)\|_{L^\infty} \LC \|\partial_x P_1\|_{L^1} + \|P_2\|_{L^1} \RC.
\]
Similarly we have
\begin{align*}
{d\over dT}\LC {1\over4}\|u(T)\|^4_{L^4} \RC & \leq \|u(T)\|^3_{L^\infty} \LC \|\partial_x P_1\|_{L^1} + \|P_2\|_{L^1} \RC,\\
{d\over dT}\LC {1\over2}\|u_Y(T)\|^2_{L^2} \RC & \leq \|u_Y(T)\|_{L^\infty} \LC \|\partial_Y(\partial_x P_1)\|_{L^1} + \|\partial_YP_2\|_{L^1} \RC,\\
{d\over dT}\LC {1\over4}\|u_Y(T)\|^2_{L^2} \RC & \leq \|u_Y(T)\|^3_{L^\infty} \LC \|\partial_Y(\partial_x P_1)\|_{L^1} + \|\partial_YP_2\|_{L^1} \RC.
\end{align*}

From \eqref{ueq} and \eqref{control xi} we know that
\beq\label{control uY}
\|u_Y(T)\|_{L^\infty} \leq {1\over2}\|\xi(T)\|_{L^\infty} \leq {1\over2}e^{A_0T}.
\eeq
Therefore in order to show that $\|u(T)\|_{H^1} + \|u(T)\|_{W^{1,4}}$ is bounded for $T < \infty$, it suffices to derive the $W^{1,1}$ bound for $\partial_x P_1$ and $P_2$. For simplicity we only consider $\|\partial_Y(\partial_x P_1)\|_{L^1}$. The other terms can be bounded the same way.

First we look for a lower bound of $|\int^Y_{\bar Y} \xi(\bar Y) \cos^4{v(\bar Y)\over 2} \ d\bar Y|$. Since we have restricted $v\in [-\pi, \pi]$ with endpoints identified, it follows that for $Y > \bar Y$
\begin{align*}
\int^Y_{\bar Y} \xi(\tilde Y) \cos^4{v(\tilde Y)\over 2} \ d\tilde Y
=&
\int^Y_{\bar Y} \xi(\tilde Y) \cos^2{v(\tilde Y)\over 2} \ d\tilde Y-\int^Y_{\bar Y} \xi(\tilde Y) \cos^2{v(\tilde Y)\over 2}\sin^2{v(\tilde Y)\over 2} \ d\tilde Y\\
\geq&
 \int_{\left\{ \tilde{Y}\in [\bar{Y}, Y],\ \left| {v(\tilde{Y})\over2} \right| \leq {\pi\over4} \right\}} \xi(\tilde Y) \cos^2{v(\tilde Y)\over2}\ d\tilde{Y} - {E}(0)\\
\geq&  \int_{\left\{ \tilde{Y}\in [\bar{Y}, Y],\ \left| {v(\tilde{Y})\over2} \right| \leq {\pi\over4} \right\}} {1\over4}\xi(\tilde Y)\ d\tilde Y - E(0) \\
\geq& {\xi^-\over4}(Y - \bar Y) -  \int_{\left\{ \tilde{Y}\in [\bar{Y}, Y],\ \left| {v(\tilde{Y})\over2} \right| \geq {\pi\over4} \right\}} {1\over4}\xi(\tilde Y)\ d\tilde Y - E(0) \\
\geq& {\xi^-\over4}(Y - \bar Y) -  \int_{\left\{ \tilde{Y}\in [\bar{Y}, Y],\ \left| {v(\tilde{Y})\over2} \right| \geq {\pi\over4} \right\}} \xi(\tilde Y)\sin^4{v(\tilde Y)\over 2}\ d\tilde Y - E(0) \\
\geq&  {\xi^-\over4}(Y - \bar{Y}) - 3\LB 2E(0)^2 -F(0) \RB - {E}(0), 
\end{align*}
where now $\xi^- = e^{-A_0T}$. Thus as before we define the kernel function
\[
\Gamma(\zeta) := \min\left\{ 1, \exp\left( K_0 - {|\zeta|\xi^-\over4} \right) \right\},
\]
with the property that
\[
\|\Gamma\|_{L^1} = {8(K_0 + 1) \over \xi^-} = 8e^{A_0T}(K_0 + 1),
\]
where $K_0 = 3\LB 2E(0)^2 -F(0) \RB + {E}(0)$. Thus upon using \eqref{pYpxP1} we have
\begin{align*}
\|\partial_Y(\partial_x P_1)\|_{L^1} \leq & \left\| u\LC {3\over8}\sin^2 v + u^2\cos^4{v\over2}  \RC\xi \right\|_{L^1} \\
& + {1\over2}\|\xi\|_{L^\infty} \left\| \Gamma \ast u\LC {3\over8}\sin^2 v + u^2\cos^4{v\over2}  \RC\xi \right\|_{L^1}\\
\leq & \LC 1 + {1\over2}\|\xi\|_{L^\infty} \|\Gamma\|_{L^1} \RC\left\| u\LC {3\over8}\sin^2 v + u^2\cos^4{v\over2}  \RC\xi \right\|_{L^1} \\
\leq & {3\over2} \LB 1 + 4 e^{2A_0T}(K_0 + 1) \RB {E}(0)^{3/2}.
\end{align*}

Lastly we try to bound $\|v\|_{L^2}$. Multiplying $v$ to the second equation of \eqref{semi} we have
\begin{align*}
{d\over dT}\LC {1\over2}\|v(T)\|^2_{L^2} \RC \leq & \LC \|u\|_{L^2} + 2\|u\|_{L^\infty}\|u\|_{L^2} \RC \|v\|_{L^2} \\
& + \LC \|P_1\|_{L^1} + \|\partial_x P_2\|_{L^1} \RC  \|v\|_{L^\infty}.
\end{align*}
The bounds for $\|P_1\|_{L^1}$ and $\|\partial_x P_2\|_{L^1}$ can be obtained in the same way as before. Thus we can prove that $\|v(T)\|_{L^2}$ remains bounded for $T < \infty$. This proves \eqref{global control} and hence completes the proof of the theorem.
\end{proof}

\subsection{Inverse transformation and global existence of solutions to \eqref{nv}}
Now we use an inverse transformation on the solution of the semi-linear system to construct
the solution to \eqref{nv}.

We define $x$ and $t$ as functions of $T$ and $Y$:
\[
x(T,Y)=\bar x(Y)+\int_0^T u^2(\tau,Y)\, d\tau, \quad t = T.
\]
So
\bel{x_T}
\frac{\partial}{\partial_T}x(T,Y)=u^2(T,Y), \qquad x(0,Y)=\bar x(Y),
\eeq
which says that $x(t,Y)$ is a characteristic.

In the rest of this section, we will show that the following function
\[
u(t,x)=u(T,Y)\qquad\hbox{if}
\qquad x=x(T,Y),\quad t=T
\]
provides a weak solution of \eqref{nv}-\eqref{ID} which satisfies the properties
listed in Theorem \ref{main}. We will proceed in several steps.

\medskip

\paragraph{{\bf Step 1.}}
We show that the image of map from $(T,Y)$ to $(t,x)$ is the entire $\R^2$.
In fact, because $\|u^2\|_{L^\infty}\leq E_0$,
\[
\bar x(Y)- E_0 T\leq x(T,Y)\leq \bar x(Y)+ E_0 T,
\]
then from the definition \eqref{Y_def} of $Y$ we have
\[
\lim_{Y\rightarrow\pm \infty} x(T,Y)=\pm \infty.
\]
So the image of continuous map $(t,Y)\rightarrow(t,x(t,Y))$ covers the entire
plane $\R^2$.
\medskip

\paragraph{{\bf Step 2.}}
Now we derive the equation
\bel{X_y}
x_Y=\xi\cos^4{\frac{v}{2}}.
\eeq
In fact, a direct calculation using \eqref{semi} implies that
\[
\LC \xi\cos^4{\frac{v}{2}} \RC_T=(u^2)_Y=(x_T)_Y=(x_Y)_T,
\]
where we have also used \eqref{x_T}. From the definition \eqref{def_lgq} we know that
$x_Y=\xi\cos^4{\frac{v}{2}}$ holds for almost every $\xi\in \R$ initially, and so
equation \eqref{X_y} is true for any $T\geq0$ and $Y\in\R$.
As a consequence, for any fixed $T$, $x(T,Y)$ is non-decreasing on $Y$.

Furthermore, for any smooth function $f$:
\bel{f_T}
f_T(T,Y)=f_t(t,x)+u^2 f_x(t,x),\quad
f_Y(T,Y)=\xi\cos^4{\frac{v}{2}}\cdot f_x(t,x).
\eeq
Similar, by equation of $u_Y$ in \eqref{ueq},  when $\cos{\frac{v}{2}}\neq0$, we have
\[
u_x=\tan{\frac{v}{2}},
\]
hence also recover all equations in \eqref{def_lgq} and \eqref{tri}.

\medskip

\paragraph{{\bf Step 3.}}
In this key step, we show that $u(x,t)=u(x(Y,T),t(T))$ is well defined although the map from $(Y,T)$
to $(x,t)$ is not one-to-one. In fact, if
\[
x(T^*,Y_1)=x(T^*,Y_2),
\]
with $Y_1<Y_2$, then
\[
x(T^*,Y)=x(T^*,Y_1) \quad\hbox{for any}\quad Y\in[Y_1, Y_2]
\]
because of the monotonicity of $x(T,Y)$ on $Y$ given in \eqref{X_y}.
Then still by \eqref{X_y},
\[
\cos{\frac{v}{2}}(T^*,Y)=0,\quad\hbox{when}\quad Y\in[Y_1,Y_2].
\]
Finally using the equation \eqref{ueq} on $u$, we know
\[
u_Y(T^*,Y)=0,\quad\hbox{when}\quad Y\in[Y_1,Y_2],
\]
hence
\[
u(T^*,Y_1)=u(T^*,Y_2),
\]
which means that $u(t,x)=u(t(T),x(T,Y))$ is well defined.
\medskip

\paragraph{{\bf Step 4.}}
Now we discuss the regularity of $u(t,x)$ and energy conservations.
Recall, by \eqref{cons E new},  energies $E(T)$ and $F(T)$ are both conservative
on $(T,Y)$ coordinates. Using this information, we prove \eqref{energy1}
and \eqref{energy2} in the Theorem \ref{main}.

For any given time $t$,
\begin{equation}\label{cons E}
\begin{split}
\mathcal{E}(0)=E(0)=E(T)
	&= \int_\R \LC u^2\cos^2{v\over2} + \sin^2{v\over2} \RC \xi \cos^2{v\over2}\,(T) \ dY\\
	&= \int_{\R\cap\{ \cos \frac{v}{2}\neq 0\}} \LC u^2\cos^2{v\over2} + \sin^2{v\over2} \RC \xi \cos^2{v\over2}\,(T)  \ dY\\
	&=\int_{\R\cap\{ \cos \frac{v}{2}\neq 0\}} \LC u^2+u_x^2 \RC (t)  \ dY\\
	&=\mathcal{E}(t)
\end{split}
\end{equation}
which proves \eqref{energy1}.

Similarly for any given time $t$,
\begin{equation}\label{cons F}
\begin{split}
\mathcal{F}(0)=F(0)=F(T)
	&= \int_\R \LC u^4 \cos^4{v\over2} + 2u^2 \cos^2{v\over2} \sin^2{v\over2} - {1\over3} \sin^4{v\over2} \RC \xi  \,(T)\ dY\\
	&=\int_{\R\cap\{ \cos \frac{v}{2}\neq 0\}} \LC u^4 \cos^4{v\over2} + 2u^2 \cos^2{v\over2} \sin^2{v\over2}  \RC \xi \,(T) \ dY\\
	&\quad-\int_{\R} {1\over3} \sin^4{v\over2}  \xi \,(T) \ dY\\
		&\leq\int_{\R\cap\{ \cos \frac{v}{2}\neq 0\}} \LC u^4 \cos^4{v\over2} + 2u^2 \cos^2{v\over2} \sin^2{v\over2}  \RC \xi  \,(T)\ dY\\
	&\quad-\int_{\R\cap\{ \cos \frac{v}{2}\neq 0\}} {1\over3} \sin^4{v\over2}  \xi  \ dY\\
	&=\int_{\R}\LC u^4 + 2u^2u^2_x - {1\over3}u^4_x \RC(t)\ dx\\
	&= \mathcal{F}(t)
\end{split}
\end{equation}
which proves \eqref{energy2}.

As a consequence, we know
$u(t,\cdot)\in W^{1,4}\cap W^{1,2}$  for any $t$.  On the other hand
\[
\frac{d u(t,x(t,Y))}{dt}=u_T<\infty.
\]
So $u(t,x)$ is H\"older continuous with exponent $3/4$ on both $t$ and $x$ by Sobolev embedding inequalities.

\medskip

\paragraph{{\bf Step 5.}}
In this part, we show that the map $t\rightarrow u(t,\cdot)$ is Lipschitz continuous under $L^4(\R)$ distance. We consider any time interval $[\tau,\tau+h]$. For any given point $(\tau,\hat x)$, we find the characteristic $x(T): \{ T\rightarrow x(T,Y)\}$ passing through $(\tau, \hat x)$, i.e. $x(\tau)=\hat x$.

Since characteristic speed $u^2$ satisfies $\|u^2\|_{L^\infty}\leq E_0$ and also by the first equation in (\ref{semi}), we have
\bes
|u(\tau+h,\hat x)-u(\tau,\hat x)|
&\leq& |u(\tau+h,\hat x)-u(\tau+h,x(\tau+h,Y))|
\nn\\[2mm]
&& +\  |u(\tau+h, x(\tau+h,Y))-u(\tau, x(\tau,Y))|\nn\\
&\leq&\sup_{|y-\hat x|\leq E_0 h}|u(\tau+h,y)- u(\tau+h,\hat x)|
+\int_\tau^{\tau+h}|\partial_x P_1 +P_2|\, dt\,.\nn
\ees
Then integrating it, we obtain
\bes
\int_{\R}|u(\tau+h,\hat x)-u(\tau,\hat x)|^4\,dx\nn
&\leq&C_1 \LB \int_{R}\left(\int_{\hat x- E_0 h}^{\hat x+ E_0 h}|u_x(\tau+h,y)|\,dy\right)^4\,dx \right.\nn\\
&&
\quad + \left. \int_{\R}\left(\int_\tau^{\tau+h}|\partial_x P_1 +P_2|\, dt\right)^4\,dx \RB \nn\\
&\leq&C_2 h^3 \LB \int_{\R} \int_{\hat x- E_0 h}^{\hat x+ E_0 h}|u_x(\tau+h,y)|^4\,dy\, dx \right. \nn\\
&&
\quad \left. + \int_{\R} \int_\tau^{\tau+h}|\partial_x P_1 +P_2|^4\, dt\, dx \RB \nn\\
&\leq& C_3 h^4 (\|u_x\|_{L^4}^4 +\|\partial_x P_1 +P_2\|_{L^4})^4\,,
\ees
where $C_i$'s are all positive constants, and $C_3$ depends on $E_0$.
Since $\|u_x\|_{L^4}$ and $\|\partial_x P_1 +P_2\|_{L^4}$ are both uniformly bounded, the map $t\rightarrow u(t,\cdot)$ is Lipschitz continuous under $L^4(\R)$ metric.

\medskip

\paragraph{{\bf Step 6.}}
Now, we prove that the function $u$ provides a weak solution of (\ref{nv}).
We denote
\[
\Lambda=[0,\infty)\times\mathbb{R}\qquad\hbox{and}\qquad \hat\Lambda=\Lambda\cap\{ (T,Y)|\cos \frac{v}{2}(T,Y)\neq 0\}.
\]

For any test function where $\phi(x,t)\in C^1_c(\Lambda)$, using  \eqref{uyt} and $u_Y=0$ when $\cos\frac{v}{2}=0$ which is given by \eqref{ueq}, we have
\bes
0&=&\iint_{\Lambda} \left\{u_{YT} \phi_T+\Big(\frac{3}{8} u \sin^2 v +(-u^3+P_1 +\partial_x P_2)\cos^4{\frac{v}{2}}\Big)\xi\phi\right\}\, dY\, dT\nonumber\\
	&=&	\iint_{\Lambda} \left\{-u_Y \phi_T+\Big(\frac{3}{8} u \sin^2 v +(-u^3+P_1 +\partial_x P_2)\cos^4{\frac{v}{2}}\Big)\xi\phi\right\}\, dY\, dT\nonumber\\
	&=&\iint_{\hat\Lambda} \left\{-u_Y \phi_T+\Big(\frac{3}{8} u \sin^2 v +(-u^3+P_1 +\partial_x P_2)\cos^4{\frac{v}{2}}\Big)\xi\phi\right\}\, dY\, dT\nonumber\\
	&=&	\iint_{\Lambda} \left\{-u_x \phi_T+(-\frac{3}{2}u u_x^2-u^3+P_1 +\partial_x P_2)\phi\right\}\, dx\, dt\nonumber\\
	&=&	\iint_{\Lambda}
	\left\{-u_x\, \bigl(\phi_t+u^2\,\phi_x)+(-\frac{3}{2}u u_x^2-u^3+P_1 +\partial_x P_2)\phi\right\}\, dx\, dt\nonumber
	\,,\label{proof_thm}
\ees
which is exactly \eqref{nv_weak}.

Now we induce the Radon measures $\{\mu_{(t)},\, t\in \R^+\}$: For any Lebesgue  measurable set $\{x\in\mathcal{A}\}$ in $\R$, supposing the corresponding pre-image set of transformation is $\{Y\in \mathcal{G}(\mathcal{A})\}$, one has
\[
\mu_{(t)}(\mathcal{A})=\int_{\mathcal{G}(\mathcal{A})} (\xi\sin^4{\frac{v}{2}})\,(t,Y)\,dY\,.
\]
Note for every $t$, $u(t,\cdot)\in W^{1,4}$, the set $\{x\,| \,\cos{\frac{v}{2}}(t,x)=0\}$ is of Lebesgue measure zero. Hence, for every $t\in \R^+$,  the absolutely continuous part of $\mu_{(t)}$ w.r.t. Lebesgue measure has density $u_x^4(t,\cdot)$ by \eqref{X_y}.

We remark that (\ref{weak_en}) is correct. In fact, by \eqref{semi},
\bes
-\int_{\R^+} \left\{ \int(\phi_t+u^2\phi_x)d\mu_{(t)}\,\right\}dt
&=&-\iint_{\Lambda}\sin^4{\frac{v}{2}}\cdot\xi\,\phi_T\,dY\,dT\label{en_law_pf}\\
&=&\iint_{\Lambda}(\sin^4{\frac{v}{2}}\cdot\xi)_T\,\phi\,dY\,dT\nonumber\\
&=&\iint_{\Lambda}4\Big(u^3 - (P_1+\partial_x P_2)\Big)\cos{\frac{v}{2}}\, \sin^3{\frac{v}{2}}\cdot\xi\,\phi\,dY\,dT\nonumber\\
&=&\iint_{\hat\Lambda}4\Big(u^3 - (P_1+\partial_x P_2)\Big)\cos{\frac{v}{2}}\, \sin^3{\frac{v}{2}}\cdot\xi\,\phi\,dY\,dT\nonumber\\
&=&\iint_{\Lambda}\Big(4u^3 u_x^3-4u_x^3 (P_1+\partial_x P_2)\Big)\phi\,dx\,dt\,.\nonumber
\ees

\paragraph{{\bf Step 7.}}
Similar as \eqref{en_law_pf}, the second energy $\nu_{(t)}(\R)$ is conserved by \eqref{cons F new}.
Finally we prove that
for almost every $t\in \R^+$, the singular part of $\nu_{(t)}$ is concentrated on the set where $u=0$.

In fact, when blowup happens, $\cos{\frac{v}{2}}=0$, so
\bel{vtu_s}
v_T=-u
\eeq
which is nonzero only when $u$ is non-zero. Hence, by similar proof as in \cite{BC2}, we can prove that for almost every $t\in \R^+$, the singular part of $\nu_{(t)}$ is concentrated on the set where $u=0$.

\begin{remark}\label{rem}
(i) Note that the result proved in Step 7 is different from the one for Camassa-Holm equation. In fact, in Camassa-Holm equation, $v$ is defined in a similar way as we do in this paper, while when singularity happens,
\[
v_T=-1
\]
which is never zero. Because of this transversality property, the energy conservative solution for the  Camassa-Holm equation has no energy concentration for almost every time.

However, for the Novikov equation, energy density $\nu_{(t)}$ might be concentrated on a set of time whose measure is not zero. When energy concentration of $\nu_{(t)}$ happens, some characteristics tangentially touch each other, then stay together for a period of time. On this piece of characteristic,  $\cos{\frac{v}{2}}=0$ and $u=0$. For the Camassa-Holm equation, when characteristics meet tangentially, they will separate immediately.

We believe that this difference is caused by the nonlinearity of the wave speed $c(u)$ for Novikov equation. In fact, $c(u)=u$ so $c'(u)\equiv 1 $ for the Camassa-Holm equation and $c(u)=u^2$ so $c'(u)=2u$ for the Novikov equation. Another nonlinear equation with cusp singularity and nonlinear wave speed is the variational wave equation
\bel{vw}
u_{tt}-c(u)(c(u)u_x)_x=0, \qquad 0<C_L<c<C_R
\eeq
for some positive constants $C_L$ and $C_R$. This equation can be derived from liquid crystal, elasticity, etcs. The global existence and uniqueness of the H\"older continuous energy conservative solution for the Cauchy problem of \eqref{vw} were established in \cite{BCZ2, BZ, HR}. Especially, for almost every time, the energy concentration happens on the set where $c'(u)=0$, cf. \cite{BZ}.

(ii) We also want to point out that in contrast to the CH case, where the energy functional $\mathcal{E}$ is conserved for almost all time, here we have from \eqref{cons E} and \eqref{cons F} an a.e. in $t$ conservation of $\mathcal{F}$ and an exact conservation of $\mathcal{E}$. This is due to the improved regularity of the solution: roughly speaking, we have $\int_{\{\cos{v\over2} = 0\}} u^2_x \ dx \leq \|u_x\|^{1/2}_{L^4} \left| \{\cos{v\over2} = 0\} \right|^{1/2} = 0$.
\end{remark}
%

\section{Uniqueness for solution of (\ref{nv})-(\ref{ID})}

In this section, we prove that the energy conservative weak solution $u(t,x)$ satisfying Definition
\eqref{def1} is unique, based on the frameworks established in \cite{BCZ, BCZ2}.

The key idea is to introduce some new energy variable $\beta$, then we prove that $u(t,x)$ satisfies a semi-linear system, which is very similar to \eqref{semi}, under independent variables $(t,\beta)$. Using the uniqueness of the solution to the new semi-linear system, we prove the uniqueness of energy conservative weak solution of \eqref{nv}-\eqref{ID}. However, as discussed in Remark \ref{rem}, solutions for Novikov equation might be much more singular than those for Camassa-Holm equation. So we need to introduce some techniques suitable for nonlinear wave speed.

Here, for any time $t$ and $\beta\in \R$, we define
$x(t,\beta)$
to be the unique point $x$ such that
\bel{xgen}
x(t,\beta)+\mu_{(t)}\bigr\{(-\infty, x)\bigl\}~\leq~ \beta~\leq~x(t,\beta)
+\mu_{(t)}\bigl\{ (-\infty, x]\bigr\}.
\eeq
Hence,
\bel{beta_def}
\beta= x(t,\beta)+\mu_{(t)}\bigr\{(-\infty, x)\bigl\} + \theta\cdot\mu_{(t)}\bigr\{x\bigl\}
\eeq
for some $\theta\in[0,1]$.

Notice  at every time where $
\mu_{(t)}$ is absolutely continuous with density $u_x^4$ w.r.t.~Lebesgue measure,
the above definition gives that
\[
\beta=x(t,\beta)+\int_{-\infty}^{x(t,\beta)}u_x^4(t,\eta)~d\eta.
\]
The next lemma, together with Lemma~\ref{lem3},
establishes the Lipschitz continuity of $x$ and $u$ as functions
of the variables $t,\beta$.
\v
\begin{lemma}\label{lem1} { Let $u=u(t,x)$ be a conservative solution of \eqref{nv}-\eqref{ID}.
Then, for every $t\ge 0$,  the  maps
$\beta\mapsto x(t,\beta)$ and $\beta\mapsto u(t,\beta)\doteq u(t, x(t, \beta))$
implicitly defined by (\ref{xgen})
are Lipschitz continuous with Lipschitz constant $1$.  The map $t\mapsto x(t,\beta)$
is also Lipschitz continuous with a constant depending only on
$\|u_0\|_{H^1}$ and $\|u_0\|_{W^{1,4}}$.
}
\end{lemma}

\begin{proof}
{\bf Step 1.} Fix any time $t\geq 0$.  Then
the map
$$x~\mapsto ~\beta(t,x)$$
is right continuous and strictly increasing.  Hence it has a well defined, continuous,
nondecreasing
inverse $\beta\mapsto x(t,\beta)$.
If  $\beta_1<\beta_2$, then
\bel{b12}
x(t,\beta_2)- \-x(t,\beta_1)+\mu_{(t)}\bigl\{ (x(t,\beta_1)\,,~x(t,\beta_2))\bigr\}
~\le~
\beta_2-\beta_1\,.
\eeq
This implies
$$x(t,\beta_2)-x(t,\beta_1)~\leq~ \beta_2-\beta_1\,,$$
showing that the map $\beta\mapsto x(t,\beta)$ is a contraction.
\v
{\bf Step 2.} To prove the
 Lipschitz continuity of the map $\beta\mapsto u(t,\beta)$, assume $\beta_1<\beta_2$.
By (\ref{b12}) it follows
\begin{equation}\label{ulip}
\begin{split}
\LV u(t, x(t, \beta_2)) - u(t, x(t, \beta_1)) \RV & \leq \int^{x(t, \beta_2)}_{x(t, \beta_1)} |u_x| \ dx \\
& \leq \LB x(t,\beta_2)- \-x(t,\beta_1)+\mu_{(t)}
 \LCB (x(t,\beta_1)\,,~x(t,\beta_2))\RCB \RB \\
& \leq (\beta_2 - \beta_1).
\end{split}
\end{equation}

{\bf Step 3.} Next, we prove the Lipschitz continuity of the map
$t\mapsto x(t,\beta)$.    Assume $x(\tau,\beta) =y$.
We recall that the family of measures $\mu_{(t)}$ satisfies the balance law (\ref{weak_en}),
where for each $t$ the source term  $4u^3 u_x^3-4u_x^3 (P_1+\partial_x P_2)$ satisfies
\bel{stb}\|4u^3 u_x^3-4u_x^3 (P_1+\partial_x P_2)\|_{L^1(\R)}~\leq~C_S\,,\eeq
for some constant $C_S$ depending only on the $H^1$ and $W^{1,4}$ norm of $u$.
Furthermore,
\bel{uib}
\|u^2\|_{L^\infty(\R)}~\leq~C_\infty~\doteq~\|u\|^2_{H^1(\R)}\,.\eeq

Therefore, for $t>\tau$ we have
$$\mu_{(t)}\bigl\{(-\infty\,,~ y-C_\infty(t-\tau))\bigr\}\leq\mu_{(\tau)}\bigl\{(-\infty\,,~ y)\bigr\}+C_S(t-\tau)\,.$$
Defining $y^-(t)\doteq y-(C_\infty+C_S)(t-\tau)$, we obtain
\begin{equation*}
\begin{split}
y^-(t)+\mu_{(t)}\LCB (-\infty\,,~ y^-(t)] \RCB & \leq y-(C_\infty+C_S)(t-\tau)+\mu_{(\tau)} \LCB (-\infty\,,~ y) \RCB+C_S(t-\tau) \\
& \leq y+\mu_{(\tau)}\LCB (-\infty\,,~ y)\RCB \leq \beta
\end{split}
\end{equation*}
%
This implies $x(t,\beta)\geq y^-(t)$ for all $t>\tau$.    A similar
argument yields
$$x(t,\beta)\leq y^+(t)\doteq y+(C_\infty+C_S)(t-\tau),$$
proving the uniform Lipschitz continuity of the map $t\mapsto x(t,\beta)$.
\end{proof}

The next lemma shows that characteristics can be uniquely determined by an integral equation
combining the characteristic equation and balance law of $\mu_{(t)}$.

\begin{lemma}\label{lem2} { Let $u=u(t,x)$ be a conservative solution of \eqref{nv}-\eqref{ID}. Then, for any $\bar y\in\R$ there exists a unique
Lipschitz continuous map  $t\mapsto x(t)$ which
satisfies both
\bel{char}
\frac{d}{dt}x(t)=u^2(t,x(t)),\qquad x(0)=\bar y
\eeq
and
\begin{equation}\label{ytx2}
\begin{split}
&{d\over dt}\LB \mu_{(t)}\LCB (-\infty, x(t))\RCB + \theta(t,\bar y)\cdot\mu_{(t)}\LCB x(t)\RCB
\RB \\
&\quad = \int_{-\infty}^{x(t)}\left[ 4u^3 u_x^3-4u_x^3 (P_1+\partial_x P_2) \right](t,x)\ dx, \quad x(0) = \bar y,
\end{split}
\end{equation}
%
for some function $\theta\in[0,1]$ and a.e.~$t\geq 0$.
In addition, for any $0\le \tau\leq t$ one has
\bel{ucar}
u(t, x(t)) - u(\tau, x(\tau)) ~=~-\int_\tau^t (\partial_xP_1 + P_2)(s, x(s))\, ds\,.\eeq
}
\end{lemma}

\begin{proof} {\bf Step 1.} Using the adapted  coordinates $(t, \beta)$, we
write the characteristic
starting at $\bar y$ in the form $t\mapsto x(t)=x(t, \beta(t))$, where
$\beta(\cdot)$ is a map to be determined.
By summing up (\ref{char}) and (\ref{ytx2}) and integrating w.r.t.~time
we obtain
\bel{ieq}\bega{l}\ds
\beta(t)~=~\ds x(t) + \mu_{(t)}\bigr\{(-\infty,x(t))\bigr\}+\theta(t)\cdot\mu_{(t)}\bigr\{x(t)\bigl\}\cr\cr
\qquad\ =~\ds\bar \beta
 + \int_0^t  \int_{-\infty}^{x(s)}[2u u_x + 4u^3 u_x^3-4u_x^3 (P_1+\partial_x P_2)  ](s,x)\, dx\,ds\,,
\enda
\eeq
with
\bel{bbeta}\bar\beta~=~\bar y + \int_{-\infty}^{\bar y}
u_{0,x}^4(x)\, dx\,.\eeq
Introducing the
function
\bel{Gdef}
G(t, \beta)~ \doteq~
\int_{-\infty}^{x(t,\beta)}[2u u_x + 4u^3 u_x^3-4u_x^3 (P_1+\partial_x P_2) ]
\, dx\,,\eeq
then
we can rewrite the equation (\ref{ieq}) in the form
\bel{bt}
\beta(t)~=~\bar\beta+\int_0^t G(s,\beta(s))\, ds\,.\eeq
The equation (\ref{bt}) is the starting point for our analysis.
In the following steps we will show that this integral equation has a unique solution
$t\mapsto \beta(t)$.
Moreover the corresponding function $t\mapsto x(t,\beta(t))$  satisfies the other two equations in the Lemma.

{\bf Step 2.} For each fixed $t\geq 0$,
the function $\beta\mapsto G(t,\beta)$
defined at (\ref{Gdef}) is
uniformly bounded and absolutely continuous.
Furthermore, by \eqref{beta_def} and \eqref{Gdef},
\[
G_\beta(t,\beta)=0\,\quad\hbox{when}\quad
 \mu_{(t)}\{x(t,\beta)\}\neq0,
 \]
 and elsewhere
\begin{eqnarray}
G_\beta=[2u u_x + 4u^3 u_x^3-4u_x^3 (P_1+\partial_x P_2) ]\, x_\beta\ \in ~[-C,\,C]
\label{Glip}
\end{eqnarray}
for some constant $C$ depending only on the $H^1$ and $W^{1,4}$ norm of $u$.
Hence the function $G$ in (\ref{Gdef}) is uniformly Lipschitz continuous w.r.t.~$\beta$.
\v
{\bf Step 3.} Thanks to the Lipschitz continuity of the function $G$, the existence of a unique
solution to the
integral equation (\ref{bt})
can be proved by a standard fixed point argument. The detail can be found in \cite{BCZ}.

 \v
\begin{figure}[htbp]
\centering
  \includegraphics[width=4in]{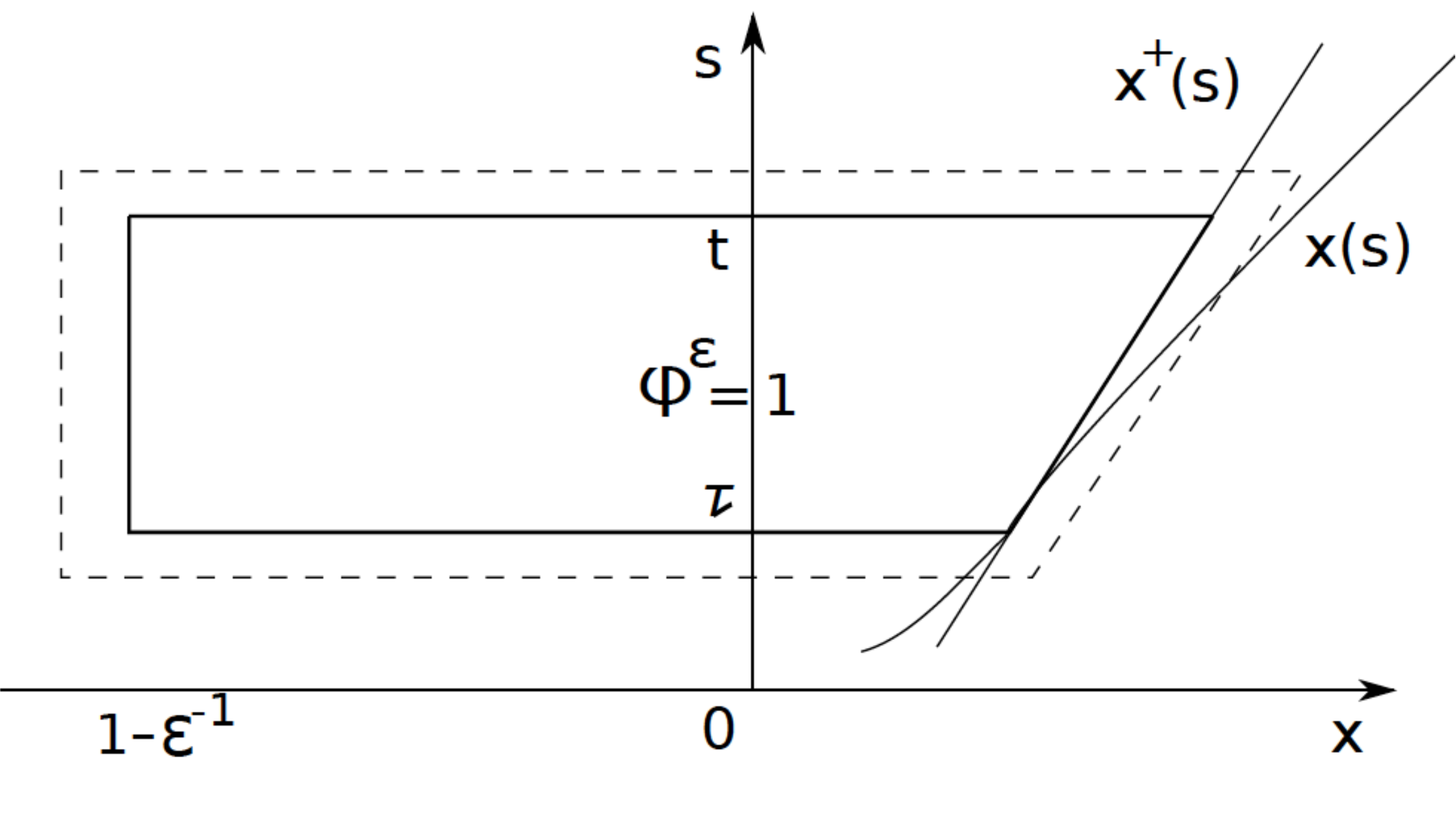}
       \caption{\small Test function $\vp^{\epsilon}$ introduced at (\ref{pen1}).  In the region outside the dashed box, $\vp^{\epsilon}=0$. }
\label{f:wa15}
\end{figure}

{\bf Step 4.} By the previous construction, the map
$t\mapsto x(t)\doteq x(t, \beta(t))$ defined by \eqref{xgen} provides the unique solution to (\ref{ieq}). Due to the Lipschitz continuity of $\beta(t)$ and $x(t)=x(t, \beta(t))$, $\beta(t)$ and $x(t)$ are differentiable almost everywhere, so we only have to consider the time where $x(t)$ is differentiable.

It suffices to show that
 (\ref{char}) holds at almost every time.
Assume, on the contrary, $\dot x(\tau)\not= u^2(\tau, x(\tau))$.
Without loss of generality, let
 \bel{ass}\dot x(\tau)~=~ u^2(\tau, x(\tau))+2\ve_0\eeq
 for some $\ve_0>0$.   The case $\ve_0 <0$ is entirely similar.
To derive a contradiction we
observe that, for all $t\in (\tau, \tau+\delta]$,
with $\delta>0$ small enough one has
\bel{xpmt}x^+(t)~\doteq~
x(\tau) +(t-\tau) [u^2(\tau, x(\tau))+\ve_0 ]~<~x(t)\,.\eeq
We also observe that if $\vp$ is
Lipschitz continuous with compact support then (\ref{weak_en}) is still true.

For any $\epsilon>0$ small, we can still use the below test functions used in \cite{BCZ}.
$$\rho^{\epsilon}(s,y)~\doteq~\left\{\bega{cl} 0 \qquad &\hbox{if}\quad y
\leq -\epsilon^{-1},\cr
(y+\epsilon^{-1}) \qquad &\hbox{if}\quad -\epsilon^{-1}\leq y\leq 1-\epsilon^{-1},\cr
1\qquad &\hbox{if}\quad 1-\epsilon^{-1} \leq y\leq x^+(s),\cr
1-\epsilon^{-1}(y-x(s))\qquad &\hbox{if}\quad  x^+(s)\leq y\leq x(s)^++\epsilon,\cr
0\qquad &\hbox{if}\quad y\geq x^+(s)+\epsilon,\enda\right.$$
\bel{timtest}\chi^\epsilon(s)~\doteq~\left\{\bega{cl} 0\qquad
&\hbox{if}\quad s\leq \tau-\epsilon,\cr
\epsilon^{-1}(s-\tau+\epsilon)\qquad &\hbox{if}\quad \tau-\epsilon\leq s\leq \tau,\cr
1\qquad &\hbox{if}\quad \tau\leq s\leq  t,\cr
1-\epsilon^{-1}(s-t) \qquad &\hbox{if}\quad t\leq s<t+\epsilon,\cr
0 \qquad &\hbox{if}\quad s\geq t+\epsilon.\enda\right.
\eeq

Define \bel{pen1}\varphi^{\epsilon}(s,y)~\doteq~\min\{ \varrho^{\epsilon}(s,y),
\,\chi^\epsilon(s)\}.\eeq

Using $\vp^{\epsilon}$ as test function  in (\ref{weak_en}) we obtain
\bel{vpe}
\int \Big[\int(\vp^{\epsilon}_t+u^2\vp^{\epsilon}_x)d\mu_{(t)}
+\int \Big(4u^3 u_x^3-4u_x^3 (P_1+\partial_x P_2)\Big)\vp^{\epsilon}
\, dx\Big]  dt~=~0.
\eeq

Suppose $t$ is sufficiently close to $\tau$, then
for $s\in [\tau+\epsilon, \, t-\epsilon]$ one has
$$0 ~=~ \vp^{\epsilon}_t + [u^2(\tau, x(\tau))+\ve_0 ] \vp^{\epsilon}_x ~
\leq~\vp^{\epsilon}_t + u^2(s,x) \vp^{\epsilon}_x\,,$$
because  $u^2(s,x)<u^2(\tau, x(\tau))+\ve_0 $ by the H\"older continuity of $u$ and $\vp^{\epsilon}_x\leq 0$.

Since the family of measures $\mu_{(t)}$ depends continuously on $t$
in the topology of weak convergence, taking the limit of (\ref{vpe}) as
 $\epsilon\to 0$, we obtain
\bel{key1}\bega{rl}
& \mu_{(t)}\bigr\{(-\infty,x^+(t)]\bigr\}\cr\cr
& \ds \ \
\geq~
\ds \mu_{(\tau)}\bigr\{(-\infty,x^+(\tau)]\bigr\}
+ \int_\tau^t\int_{-\infty}^{x^+(s)}\Big(4u^3 u_x^3-4u_x^3 (P_1+\partial_x P_2)\Big)\,dxds
\cr\cr
&\ \ =~\ds\mu_{(\tau)}\bigr\{(-\infty,x(\tau)]\bigr\}
+ \int_\tau^t\int_{-\infty}^{x(s)}\Big(4u^3 u_x^3-4u_x^3 (P_1+\partial_x P_2)\Big)\,dxds+ o_1(t-\tau).
\enda
\eeq

Notice that the last term is a higher order infinitesimal,
 satisfying
${o_1(t-\tau)\over t-\tau}\to 0$ as $t\to \tau$.
Indeed
\begin{equation*}
\begin{split}
|o_1(t-\tau)| & = \LV \int_\tau^t \int_{x^+(s)}^{x(s)}\LC 4u^3 u_x^3-4u_x^3 (P_1+\partial_x P_2)\RC \,dx ds\RV \\
& \leq \LN \LB 4u^3-4 (P_1+\partial_x P_2)\RB \RN_{L^\infty} \cdot \int_\tau^t \int_{x^+(s)}^{x(s)}\, |u_x^3|\,dxds \\
& \leq \LN \LB 4u^3-4 (P_1+\partial_x P_2)\RB \RN_{L^\infty} \cdot \int_\tau^t (x(s)-x^+(s))^{1/4}\, \|u_x(s,\cdot)\|^3_{L^4}\,ds \\
& \leq C(t-\tau)^{5/4}.
\end{split}
\end{equation*}
%

For $t$ sufficiently close to $\tau$, we have
\bel{3.19}
\bega{lll}
\ds\beta(t)&\ds=&\ds x(t)+\mu_{(t)}\bigr\{(-\infty,x(t))\bigr\}+\theta(t)\cdot\mu_{(t)}\bigr\{x(t)\bigl\}
\cr\cr
&\ds \geq&\ds x(\tau)+(t-\tau)[u^2(\tau,x(\tau))+\ve_0 ]+\mu_{(t)}\bigr\{(-\infty,x^+(t)]\bigr\}
\cr\cr
&\ds \geq&\ds x(\tau)+(t-\tau)[u^2(\tau,x(\tau))+\ve_0 ]+
\mu_{(t)}\bigr\{(-\infty,x^+(\tau)]\bigr\}\cr\cr
&&\qquad \ds +\int_\tau^t\int_{-\infty}^{x(s)} \Big(4u^3 u_x^3-4u_x^3 (P_1+\partial_x P_2)\Big)\, dxds+ o_1(t-\tau)
\enda
\eeq
where we also use \eqref{key1} and \eqref{xpmt}.

On the other hand, by (\ref{Gdef}) and (\ref{bt}) a linear approximation yields
\bel{3.18}\beta(t) ~=~\beta(\tau) +(t-\tau) \left[ u^2(\tau, x(\tau))
 +\int_{-\infty}^{x(\tau)}\Big(4u^3 u_x^3-4u_x^3 (P_1+\partial_x P_2)\Big)\, dx\right]+ o_2(t-\tau)\,\eeq
with ${o_2(t-\tau)\over t-\tau}\to 0$ as $t\to \tau$.

Combining (\ref{3.18}) and (\ref{3.19}), we find
\beq\bega{l}
\ds\beta(\tau) +(t-\tau) \left[ u^2(\tau, x(\tau))
 +\int_{-\infty}^{x(\tau)}\Big(4u^3 u_x^3-4u_x^3 (P_1+\partial_x P_2)\Big)\, dx\right]+ o_2(t-\tau)\cr\cr
\quad \geq\ds x(\tau)+(t-\tau)[u^2(\tau,x(\tau))+\ve_0 ]+
\mu_{(t)}\bigr\{(-\infty,x^+(\tau)]\bigr\}\cr\cr
\qquad \ds +\int_\tau^t\int_{-\infty}^{x(s)} \Big(4u^3 u_x^3-4u_x^3 (P_1+\partial_x P_2)\Big)\, dxds+ o_1(t-\tau).
\enda
\eeq
Subtracting common terms, dividing both sides by $t-\tau$ and
letting $t\to \tau$, we achieve a contradiction.  Therefore, (\ref{char}) must hold.
As a consequence, (\ref{ucar}) is proved by (\ref{char}) and (\ref{ieq}).

\v
{\bf Step 5.}  We now prove (\ref{ucar}).
By (\ref{nv_weak}), for every test function $\phi\in \C^\infty_c(\R^2)$ one has
\bel{CHW1}
\int_0^\infty \int\left[u\phi_t +{u^3\over 3}\phi_x - (\partial_xP_1 + P_2)\phi\right]\, dxdt
+ \int u_0(x)\phi(0,x)\, dx~=~0\,.\eeq
Given any $\vp\in\C^\infty_c$, let $\phi=\vp_x$.
Since the map $x\mapsto u(t,x)$ is absolutely continuous, we can
integrate by parts w.r.t.~$x$ and obtain
\bel{CHW2}
\int_0^\infty \int\left[u_x\vp_t +u^2u_x\vp_x + (\partial_xP_1 + P_2)\vp_x\right]\, dxdt
+ \int u_{0,x}(x)\vp(0,x)\, dx~=~0\,.\eeq

By an
approximation argument, the identity (\ref{CHW2}) remains
valid for any test function $\vp$ which is
Lipschitz continuous with compact support.
For any $\epsilon>0$ sufficiently small, we thus consider the functions
$$\varrho^{\epsilon}(s,y)~\doteq~\left\{\bega{cl} 0 \qquad &\hbox{if}\quad y\leq -
\epsilon^{-1},\cr
y+\epsilon^{-1} \qquad &\hbox{if}\quad -\epsilon^{-1}\leq y\leq 1-\epsilon^{-1},\cr
1\qquad &\hbox{if}\quad 1-\epsilon^{-1} \leq y\leq x(s),\cr
1-\epsilon^{-1}(y-x(s))\qquad &\hbox{if}\quad  x(s)\leq y\leq x(s)+\epsilon,\cr
0\qquad &\hbox{if}\quad y\geq x(s)+\epsilon,\enda\right.$$

\bel{pen}\psi^{\epsilon}(s,y)~\doteq~\min\{ \varrho^{\epsilon}(s,y),\,\chi^\epsilon(s)\},\eeq
with $\chi^\epsilon(s)$ as in (\ref{timtest}).
We now use the test function
$\vp=\psi^{\epsilon}$  in (\ref{CHW2}) and let
$\epsilon\to 0$. Observing that the function $(\partial_xP_1 + P_2)$
is continuous, we obtain
\bel{Inteq}\bega{rl}
\ds \int_{-\infty}^{x(t)} u_x(t,x)\, dx&=~\ds\int_{-\infty}^{x(\tau)}u_x(\tau,x)\, dx
- \int_\tau^t (\partial_xP_1 + P_2)(s, x(s))\,ds\cr\cr
&\ds\qquad +\lim\limits_{\epsilon\to 0}\int_{\tau-\epsilon}^{t+\epsilon}\int_{x(s)}
^{x(s)+\epsilon} u_x(\psi^{\epsilon}_t+u^2\psi^{\epsilon}_x) dx ds\,.\cr\cr
\enda\eeq
To complete the proof it suffices to show that the last term on the right hand side of
(\ref{Inteq}) vanishes. Since $u_x\in \L^2$, Cauchy's inequality yields
\bel{lexx}
\bega{l}
\ds\left|\int_{\tau-\epsilon}^{t+\epsilon}\int_{x(s)}^{x(s)+\epsilon}
u_x(\psi^{\epsilon}_t+u^2\psi^{\epsilon}_x) dx ds\right|\cr\cr
\qquad \leq~\ds\int_{\tau-\epsilon}^{t+\epsilon}\left(\int_{x(s)}^{x(s)+\epsilon}
 |u_x|^2dx\right)^{1/2}\left(\int_{x(s)}^{x(s)+\epsilon}
(\psi^\epsilon_t+u^2\psi^\epsilon_x)^2dx\right)^{1/2}ds\,.\cr\cr
\enda
\eeq
We only have to show that
\[
\int_{x(s)}^{x(s)+\epsilon}
(\psi^\epsilon_t+u^2\psi^\epsilon_x)^2 dx\rightarrow 0,\qquad \hbox{as} \qquad\epsilon\rightarrow 0.
\]
In fact, for every time $s\in [\tau-\epsilon,t+\epsilon]$ by construction we have
$$\psi^{\epsilon}_x(s,y)~=~\epsilon^{-1},
\qquad \psi^{\epsilon}_t(s,y)+u^2(s, x(s))\psi^{\epsilon}_x(s,y)~=~0,
$$
when $x(s)<y<x(s)+\epsilon$. This implies
\bel{luxx}\bega{l}\ds\qquad\int_{x(s)}^{x(s)+\epsilon}
|\psi^{\epsilon}_t(s,y)+u^2(s,y)\psi^{\epsilon}_x(s,y)|^2dy\cr\cr
\qquad \quad =\ds\epsilon^{-2}\int_{x(s)}^{x(s)+\epsilon}
|u^2(s,y)- u^2(s, x(s))|^2dy\cr\cr
\qquad \quad \leq~\ds \epsilon^{-1}\cdot
\left(\max_{x(s)\leq y\leq x(s)+\epsilon}|u^2(s,y)- u^2(s, x(s))|\right)^2
\cr\cr
\qquad \quad \leq~\epsilon^{-1}\cdot
\left(\int_{x(s)}^{x(s)+\epsilon} |2uu_x(s,y)|\, dy\right)^2\cr\cr
\qquad \quad \leq~\epsilon^{-1}\cdot
\left(2\|u(s)\|_{\L^\infty}\epsilon^{3/4}\cdot \|u_x(s)\|_{\L^4}\right)^2~\leq
4\|u(s)\|_{W^{1,4}}^4\epsilon^{1/2}\rightarrow0,
\enda\eeq
as $\epsilon\rightarrow 0$.
\v
{\bf Step 6.} The uniqueness of the solution $t\mapsto x(t)$ now becomes clear.
\end{proof}

\v
Relying on (\ref{ucar}) we can now show the Lipschitz continuity of $u$ w.r.t.~$t$, in the
auxiliary coordinate system.
\begin{lemma}\label{lem3} { Let $u=u(t,x)$ be a conservative solution of Novikov equation.
Then the  map $(t,\beta)\mapsto u(t,\beta)\doteq u(t, x(t, \beta))$
is Lipschitz continuous, with a constant depending only on the norm $\|u_0\|_{H^1}$ and $\|u_0\|_{W^{1,4}}$.  }
\end{lemma}
\v
\begin{proof}
Using  (\ref{ulip}), (\ref{bt}), and  (\ref{ucar}), and we obtain
\begin{equation*}
\begin{split}
& \LV u(t, x(t,\bar\beta))- u(\tau,\bar\beta) \RV \\
& \quad \leq \LV u(t, x(t,\bar\beta))- u(t,x(t,\beta(t))) \RV + \LV u(t, x(t,\beta(t)))- u(\tau, x(\tau,\beta(\tau))) \RV \\
& \quad \leq {1\over2}|\beta(t)-\bar\beta| + C(t-\tau)\\
& \quad \leq C(t-\tau)
\end{split}
\end{equation*}
where $C$ is a constant depending only on the norm $\|u_0\|_{H^1}$ and $\|u_0\|_{W^{1,4}}$.
\end{proof}

The next result shows that the solutions $\beta(\cdot)$ of (\ref{bt}) depend Lipschitz
continuously on the initial data.

\begin{lemma}\label{lem4}  { Let $u$ be a conservative solution to
Novikov equation. Call $t\mapsto \beta(t;\tau,\bar \beta)$ the solution to the
integral equation
\bel{bt2}
\beta(t)~=~\bar \beta+\int_\tau^t G(\tau, \beta(\tau))\, d\tau,
\eeq
with $G$ as in (\ref{Gdef}).    Then there exists a constant $C$
such that, for any two initial data
$\bar \beta_1,\bar\beta_2$ and any $t,\tau\geq 0$ the corresponding solutions satisfy
\bel{bil}
|\beta(t;\tau,\bar\beta_1)-
\beta(t;\tau,\bar\beta_2)|~\leq~e^{C|t-\tau|}\,|\bar\beta_1-\bar\beta_2|.\eeq
}
\end{lemma}
\v
\begin{proof} The proof is easy because of the Lipschitz continuity of $G$ with respect to $\beta$.
We leave it to the reader.
\end{proof}

\v
\begin{lemma}\label{lem5}
{ Assume $u\in H^1(\R)\cap W^{1,4}(\R)$. Then ${P_i}_x$, $i=1$ or $2$, is absolutely continuous  and satisfies
\bel{P1xx}
{P_1}_{xx}~=~{P_1}-\left(\frac{3}{2}u u_x^2+u^3\right).\eeq
and
\bel{P2xx}
{P_2}_{xx}~=~{P_2}-\frac{3}{2}u_x^3.\eeq}
\end{lemma}

\begin{proof} The function $\phi(x)=e^{-|x|}/2$ satisfies the distributional identity
$$D^2_x\phi~=~\phi -\delta_0\,,$$
where $\delta_0$ denotes a unit Dirac mass at the origin.
Therefore, for every function $f\in \L^1(\R)$, the convolution satisfies
$$D^2_x(\phi*f)~=~\phi*f - f\,.$$
Choosing $f$ to be  $\frac{3}{2}u u_x^2+u^3$ and  $u_x^3$ we obtain the result.
\end{proof}

\v

\section{Proof of uniqueness}

The proof will be established in several steps.
\v
{\bf Step 1.}
By Lemmas \ref{lem1} and \ref{lem3},  the map $(t,\beta)\mapsto (x,u)(t,\beta)$ is
Lipschitz continuous.   An entirely similar argument shows that the maps
$\beta\mapsto G(t,\beta)\doteq
G(t,x(t,\beta))$
and $\beta\mapsto {P_i}_x(t,\beta)\doteq {P_i}_x(t, x(t,\beta))$ are also Lipschitz continuous.
By Rademacher's theorem, the partial derivatives
$x_t$,  $x_\beta$, $u_t$, $u_\beta$, $G_\beta$, $ {P_1}_{x,\beta}$ and $ {P_2}_{x,\beta}$ exist almost everywhere.
Moreover, a.e.~point
$(t,\beta)$ is a Lebesgue point for these derivatives.
Calling $t\mapsto \beta(t, \bar \beta)$ the unique solution to
the integral equation (\ref{bt}), by Lemma~\ref{lem4} for a.e.~$\bar \beta$  the following holds.
\v
\begi
\item[{\bf (GC)}] For a.e.~$t> 0$, except a measure zero set $\mathcal{N}\in \R^+$, the point
$(t, \beta(t,\bar\beta))$ is a Lebesgue point for the partial derivatives
$x_t,  x_\beta, u_t, u_\beta, G_\beta, P_{x,\beta}$.
\endi
\v
If (GC) holds, we then say that $t\mapsto
\beta(t,\bar\beta)$ is a {\bf good characteristic}.
\v
{\bf Step 2.}
We seek an ODE describing how the quantities
$u_\beta$ and $x_\beta$
vary along a good characteristic.
As in Lemma~\ref{lem4}, we denote by $t\mapsto \beta(t;\tau,\bar\beta)$ the solution to
(\ref{bt2}).  If $\tau,t\notin{\mathcal N}$, assuming that $\beta(\cdot; \tau,\bar \beta)$
is a good characteristic, differentiating (\ref{bt2}) w.r.t.~$\bar \beta$ we find
\bel{bode}
{\partial\over\partial\bar\beta} \beta(t;\tau,\bar\beta)~=~
1+\int_\tau^t G_\beta(s,\beta(s;\tau,\bar\beta))
\cdot{\partial\over\partial\bar\beta} \beta(s;\tau,\bar\beta)\,ds
\eeq
Next, differentiating w.r.t.~$\bar\beta$ the identity
$$x(t, \beta(t;\tau,\bar\beta))~=~x(\tau, \bar\beta)+\int_\tau^t
u^2(s,\,x(s,\beta(t;\tau,\bar\beta)))\, ds$$
 we obtain
\bel{xbode}
x_\beta(t, \beta(t;\tau,\bar\beta))\cdot
{\partial\over\partial\bar\beta} \beta(t;\tau,\bar\beta)~=~x_\beta(\tau,\bar\beta)
+\int_\tau^t (u^2)_\beta(s,\beta(s;\tau,\bar\beta))\cdot
{\partial\over\partial\bar\beta} \beta(s;\tau,\bar\beta)\,ds.
\eeq
Finally, differentiating w.r.t.~$\bar\beta$ the identity (\ref{ucar}), we obtain
\bel{ubode}
u_\beta(t, \beta(t;\tau,\bar\beta))\cdot
{\partial\over\partial\bar\beta} \beta(t;\tau,\bar\beta)~=~u_\beta(\tau,\bar\beta)
+\int_\tau^t (\partial_xP_1 + P_2)_\beta(s,\beta(s;\tau,\bar\beta))\cdot
{\partial\over\partial\bar\beta} \beta(s;\tau,\bar\beta)\,ds\,.\eeq

Combining (\ref{bode})--(\ref{ubode}),  we thus obtain the system of ODEs
\bel{ODES}\left\{\bega{cl}\ds
{d\over dt}\left[{\partial\over\partial\bar\beta} \beta(t;\tau,\bar\beta)
\right]&\ds=~G_\beta(t,\beta(t;\tau,\bar\beta))
\cdot{\partial\over\partial\bar\beta} \beta(t;\tau,\bar\beta),
\cr
\cr
\ds{d\over dt}\left[x_\beta(t, \beta(t;\tau,\bar\beta))\cdot
{\partial\over\partial\bar\beta} \beta(t;\tau,\bar\beta)
\right]&\ds=~(u^2)_\beta(t,\beta(t;\tau,\bar\beta))\cdot
{\partial\over\partial\bar\beta} \beta(t;\tau,\bar\beta),
\cr
\cr
\ds{d\over dt}\left[u_\beta(t, \beta(t;\tau,\bar\beta))\cdot
{\partial\over\partial\bar\beta} \beta(t;\tau,\bar\beta)
\right]&\ds=~(\partial_xP_1 + P_2)_{\beta}(t,\beta(t;\tau,\bar\beta))\cdot
{\partial\over\partial\bar\beta} \beta(t;\tau,\bar\beta).\enda\right.\eeq
In particular, the quantities within square brackets on the left hand sides
of (\ref{ODES}) are absolutely continuous.
{}From (\ref{ODES}), using Lemma~\ref{lem5} along a good characteristic we obtain
\bel{xubt}\left\{\bega{rl}\ds
{d\over dt}
x_\beta+G_\beta x_\beta& =~2uu_\beta\,,\cr\cr
\ds{d\over dt}u_\beta+G_\beta u_\beta&\ds =~(-\partial_xP_1 - P_2)_x\,{x_\beta}\cr
&\ds=~\left[-P_1-\partial_x P_2+\frac{3}{2}u u_x^2+u^3\right]x_\beta
\cr
&\ds
=\left[-P_1-\partial_x P_2+u^3\right]x_\beta+\frac{3}{2}u (x_\beta(1-x_\beta))^{1/2},
\enda\right.\eeq
where the first equation is obtained by the first two equations in $(\ref{ODES})$ and
 the second equation is obtained by the first  and third equations in $(\ref{ODES})$

\v
{\bf Step 3.} We now go back to the original
$(t,x)$ coordinates and derive an evolution equation
for the partial derivative
$u_x$ along a ``good" characteristic curve.

On any ``good" characteristic (GC): $x=\Gamma(t)$, we denote the set
\[
\mathcal{T}_s=\left\{t\in \R^+ \,\Big|\, |u_x(t,\Gamma(t))|<\infty\right\}
\]
which is equivalent to the set of time when $x_\beta>0$ on $x=\Gamma(t)$. We further define
\[
\mathcal{T}_u= \left\{t\in\R^+, t\not\in\mathcal{T}_s\right\}\,,
\]
\[
\mathcal{T}_{ui}= \hbox{inner point of } \mathcal{T}_u\,, \quad \mathcal{T}_{ub}=\hbox{boundary point of } \mathcal{T}_u = \overline{\mathcal{T}_u} \backslash \mathcal{T}_{ui}.
\]
and it is clear that
\[
\mathcal{T}_{ub}=\hbox{boundary point of } \mathcal{T}_s = \mathcal{T}_{sb}\,.
\]
Hence we have
\[
\R^+=\mathcal{T}_s\cup\mathcal{T}_u=\mathcal{T}_s\cup\mathcal{T}_{ui}\cup
\mathcal{T}_{ub}\, .
\]

The solution along the good characteristic will be constructed as follows. First, if $t\in\mathcal{T}_{ui}$, we define
\bel{vu_ui}
v\bigl(t,\Gamma(t)\bigr)=\pi,\qquad u\bigl(t,\Gamma(t)\bigr)=0\,.
\eeq
Here $u=0$ because an inner point of $\mathcal{T}_u$ only possibly
exists on the set of time when $u=0$ along any fixed GC.

Fix a point $(\tau,\bar x)$ with $\tau\in{\mathcal T}_s$, on which
$x_\beta>0$.
Assume that $\bar x$ is a Lebesgue point for the map  $x\mapsto u_x(\tau, x)$.
Let  $\bar \beta$ be such that $\bar x = x(\tau,\bar\beta)$
and assume that $t\mapsto \beta(t;\tau,\bar\beta)$ is a {\em good characteristic},
so that (GC) holds. We observe that
$$u_x^4(\tau,x)~=~{1\over x_\beta(\tau,\bar\beta)}-1~\geq~0\qquad\hbox{when}\qquad x_\beta(\tau,\bar\beta)~>~0\,.$$
As long as $x_\beta>0$, along the characteristic through $(\tau,\bar x)$ the
partial derivative $u_x$ can be computed as
\bel{uxx}
u_x\Big(t,x(t,\beta(t;\tau,\bar\beta))\Big)~=~{u_\beta(t,\beta(t;\tau,\bar\beta))\over
x_\beta(t,\beta(t;\tau,\bar\beta))}\,.\eeq
Using the two ODEs (\ref{xbode})-(\ref{ubode})
describing the evolution of $u_\beta$ and $x_\beta$, we conclude that the map
$t\mapsto u_x(t, x(t,\beta(t;\tau\bar\beta)))$ is absolutely continuous
(as long as $x_\beta\not=0$) and satisfies
\bel{u_xt}\bega{l}\ds
{d\over dt} u_x(t, x(t,\beta(t;\tau\bar\beta)))~=\ds~{d\over dt}\left(
{u_\beta}\over{x_\beta}\right)\cr\cr
\ds
\quad =~
\frac{x_\beta\left\{[P_1+\partial_x P_2+\frac{3}{2}u (\frac{1}{x_\beta}-1)^{1/2}+u^3]x_\beta-u_\beta G_\beta\right\}-u_\beta
\left\{2uu_\beta-x_\beta G_\beta\right\}}{x_\beta^2}\cr\cr
 \ds \quad =~-P_1-\partial_x P_2+\frac{3}{2}u (\frac{1}{x_\beta}-1)^{1/2}+u^3-\frac{u_\beta G_\beta}{x_\beta}-2u\frac{u_\beta^2}{x_\beta^2}+\frac{u_\beta G_\beta}{x_\beta}
\cr\cr
 \ds
\quad =~-P_1-\partial_x P_2+u^3-\frac{1}{2}u (\frac{1}{x_\beta}-1)^{1/2}.
\enda\eeq
In turn, as long as $x_\beta>0$ this implies
\bel{arctan}\bega{l}\ds
{d\over dt}\arctan u_x(t, x(t,\beta(t;\tau\bar\beta)))~=~{1\over 1+u_x^2}\cdot {d\over dt} u_x
\cr\cr
\ds\qquad = \left(P_1+\partial_x P_2-\frac{3}{2}u (\frac{1}{x_\beta}-1)^{1/2}-u^3-\frac{u_\beta^2}{x_\beta^2}\right)\frac{x_\beta^{1/2}}{x_\beta^{1/2}+(1-x_\beta)^{1/2}}
\cr\cr
\ds\qquad=~\left(-P_1-\partial_x P_2+u^3 \right)\frac{x_\beta^{1/2}}{x_\beta^{1/2}+(1-x_\beta)^{1/2}}-
\frac{1}{2}u\frac{(1-x_\beta)^{1/2}}{x_\beta^{1/2}+(1-x_\beta)^{1/2}}
\enda\eeq
\v
Now we consider the function
\beq
v~\doteq~\left\{
\bega{cl}
2\arctan u_x \qquad  &\hbox{if}\quad 0<x_\beta\leq 1,\\
\pi \qquad  &\hbox{if}\quad x_\beta=0.
\enda\right.
\eeq
Observe that this implies
\bel{a22}\frac{x_\beta^{1/2}}{x_\beta^{1/2}+(1-x_\beta)^{1/2}}~=~{1\over 1+u_x^2}~=~\cos^2{v\over 2}\,,
\eeq
and
\bel{a22_2}
\frac{(1-x_\beta)^{1/2}}{x_\beta^{1/2}+(1-x_\beta)^{1/2}}~=~{u_x^2
\over 1+u_x^2}~=~\sin^2{v\over 2}\,.\eeq
In the following,  $v$ will be regarded
as a map taking values in the unit circle $\S\doteq [-\pi,\pi]$
with endpoints identified.
We claim that, along each good characteristic, the map $t\mapsto v(t)\doteq
v(t, x(t,\beta(t;\tau\bar\beta)))$ is absolutely continuous and satisfies
\bel{Vt}
{d\over dt}v(t)~=~2\left(-P_1-\partial_x P_2+u^3 \right)\cos^2{v\over 2}-
u\sin^2{v\over 2}\,.\eeq
Indeed, denote by $x_\beta(t)$, $u_\beta(t)$ and $u_x(t)= u_\beta(t)/x_\beta(t)$
the
values of $x_\beta$, $u_\beta$, and $u_x$ along this particular characteristic.
Hence we have $x_\beta(t)>0$ when $t\in\mathcal{T}_s$ and $x_\beta(t)=0$ when $t\in\mathcal{T}_{ui}$. As a consequence, when  $t\in\mathcal{T}_{ui}$, by \eqref{Vt},
\[
0=\frac{d}{dt}v(t)=-u,
\]
hence $u=0$, which agrees with \eqref{vu_ui}.

If $\tau$ is any time where $x_\beta(\tau)>0$,
we can find a neighborhood
$I=[\tau-\delta, \tau+\delta]$ such that  $x_\beta(t)>0$ on $I$.
By (\ref{arctan}) and (\ref{a22}), $v= 2\arctan(u_\beta/x_\beta)$
is absolutely continuous  restricted to $I$ and satisfies (\ref{Vt}).
To prove our claim, it thus remains to show that $t\mapsto v(t)$
is continuous on the set $\mathcal{T}_{ub}$ of times where $x_\beta(t)=0$.

Suppose $t_0 \in \mathcal{T}_{ub} = \mathcal{T}_{sb}$, which implies $x_\beta(t_0)=0$. Take a sequence $\{t_n\}$ in $\mathcal{T}_s$ with limiting point $t_0$. By definition of $\mathcal{T}_s$ we have $x_\beta(t_n)>0$. Moreover \eqref{xubt} indicates that $x_\beta(t)$ is Lipschitz continuous in $\mathcal{T}_s$. Therefore from the identity
\beq
u_x^4(t_n)~=~{{1-x_\beta(t_n)}\over{x_\beta(t_n)}}\,,\eeq
it follows that as $t_n \to t_0$, $x_\beta(t_n)\to 0$ and $u_x^4(t_n)\to \infty$.   This implies $v(t)=2\arctan u_x(t)$
$\to \pm\pi$.  Since we identify the points $\pm\pi$, this establishes the
continuity of $v$ for all $t\geq 0$, proving our claim.
\v
{\bf Step 4.}  Let now $u=u(t,x)$ be a conservative solution.
As shown by the previous analysis, in terms of the variables $t,\beta$ the quantities
$x,u,v$ satisfy the semilinear system

\bel{xuv}\left\{\bega{cl}\ds
{d\over dt} \beta(t, \bar\beta)&=~G(t, \beta(t, \bar\beta)),\cr\cr
\ds {d\over dt} x(t,\beta(t,\bar\beta))&=~u^2(t,\beta(t,\bar\beta)),\cr\cr
\ds {d\over dt} u(t,\beta(t,\bar\beta))&=~\partial_xP_1 + P_2,\cr\cr
\ds {d\over dt} v(t,\beta(t,\bar\beta))&=~
2\left(-P_1-\partial_x P_2+u^3 \right)\cos^2{v\over 2}-
u\sin^2{v\over 2}\,.
\enda\right.\eeq
We recall that $P_1$, $P_2$ and $G$ were defined at (\ref{P12}) and (\ref{Gdef}), respectively.
Furthermore, $P_1$, $P_2$, $\partial_x P_1$ and $\partial_x P_2$ admits representations in terms of the variable $\beta$, namely
{\small
\bel{Pc1}\left.
\bega{l}\ds
P_1(x(\beta))
= {1\over 2}
\int_{-\infty}^\infty
\exp\left\{ -\left|
\int_\beta^{\beta'} \frac{\cos^4{v(\beta')\over 2}}{\sin^4{v(\beta')\over 2}+\cos^4{v(\beta')\over 2}}
\,ds\right|\right\} \cdot\cr\cr
\ds\qquad\qquad\qquad
\left[ \frac{3}{2}u(\beta')\frac{\sin^2{v(\beta')\over 2}\cos^2{v(\beta')\over 2}}{\sin^4{v(\beta')\over 2}+\cos^4{v(\beta')\over 2}}+
u^3\frac{\cos^4{v(\beta')\over 2}}{\sin^4{v(\beta')\over 2}+\cos^4{v(\beta')\over 2}}
\right]\,d\beta',
\enda\right.\eeq

\bel{Pc2}\left.
\bega{l}\ds
P_2(x(\beta))
= {1\over 2}
\int_{-\infty}^\infty
\exp\left\{ -\left|
\int_\beta^{\beta'} \frac{\cos^4{v(\beta')\over 2}}{\sin^4{v(\beta')\over 2}+\cos^4{v(\beta')\over 2}}
\,ds\right|\right\} \cdot\cr\cr
\ds\qquad\qquad\qquad
\left[ \frac{1}{2}\frac{\sin^3{v(\beta')\over 2}\cos{v(\beta')\over 2}}{\sin^4{v(\beta')\over 2}+\cos^4{v(\beta')\over 2}}
\right]\,d\beta',
\enda\right.\eeq

\bel{Pcx1}\left.
\bega{l}\ds
\partial_x P_1(x(\beta))
= {1\over 2}
\left(
\int_\beta^\infty - \int_{-\infty}^\beta\right)
\exp\left\{ -\left|
\int_\beta^{\beta'} \frac{\cos^4{v(\beta')\over 2}}{\sin^4{v(\beta')\over 2}+\cos^4{v(\beta')\over 2}}
\,ds\right|\right\} \cdot\cr\cr
\ds\qquad\qquad\qquad\quad
\left[ \frac{3}{2}u(\beta')\frac{\sin^2{v(\beta')\over 2}\cos^2{v(\beta')\over 2}}{\sin^4{v(\beta')\over 2}+\cos^4{v(\beta')\over 2}}+
u^3\frac{\cos^4{v(\beta')\over 2}}{\sin^4{v(\beta')\over 2}+\cos^4{v(\beta')\over 2}}
\right]\,d\beta',
\enda\right.\eeq

\bel{Pcx2}\left.
\bega{l}\ds
\partial_x P_2(x(\beta))
= {1\over 2}
\left(
\int_\beta^\infty - \int_{-\infty}^\beta\right)
\exp\left\{ -\left|
\int_\beta^{\beta'} \frac{\cos^4{v(\beta')\over 2}}{\sin^4{v(\beta')\over 2}+\cos^4{v(\beta')\over 2}}
\,ds\right|\right\} \cdot\cr\cr
\ds\qquad\qquad\qquad\quad
\left[ \frac{1}{2}\frac{\sin^3{v(\beta')\over 2}\cos{v(\beta')\over 2}}{\sin^4{v(\beta')\over 2}+\cos^4{v(\beta')\over 2}}
\right]\,d\beta',
\enda\right.\eeq

}
For every $\bar\beta\in \R$ we have the initial condition
\bel{ico}\left\{\bega{cl}\ds\beta(0,\bar\beta)&=~\bar\beta,\cr
x(0,\bar\beta)&=~x(0,\bar\beta),\cr
u(0,\bar\beta)&=~u_0(x(0,\bar\beta)),\cr
v(0,\bar\beta)&=~2\arctan u_{0,x}(x(0,\bar\beta)).\enda\right.
\eeq

To see the Lipschitz continuity of all coefficients, we only need to note one fact that
\[
\sin^4\frac{v}{2}+\cos^4 \frac{v}{2}=\sin^4\frac{v}{2}+(1-\sin^2 \frac{v}{2})^2\geq\frac{1}{2}.
\]
As a consequence,
the Cauchy problem (\ref{xuv}), (\ref{ico}) has a unique
solution, globally defined for all $t\geq 0$, $x\in \R$.
\v
{\bf Step 5.}
To complete the proof of uniqueness,   consider two conservative solutions
$u,\tilde u$ of the Novikov equation (\ref{nv})
with the same initial data $u_0\in H^1(\R)\cap W^{1,4}(\R)$.
For a.e.~$t\geq 0$  the
corresponding  Lipschitz continuous
maps $\beta\mapsto x(t,\beta)$, $\beta\mapsto \tilde x(t,\beta)$
are strictly increasing.  Hence they have continuous inverses, say
$x\mapsto \beta^* (t,x)$, $x\mapsto \tilde \beta^* (t,x)$.

By the previous analysis, the map    $(t,\beta)\mapsto (x,u,v)(t,\beta)$
is uniquely determined by the initial data $u_0$.  Therefore
$$x(t,\beta)~=~\tilde x(t,\beta),\qquad\qquad u(t,\beta)=\tilde u(t,\beta).$$
In turn, for a.e.~$t\geq 0$ this implies
$$u(t,x) ~=~u(t,\beta^*(t,x))~=~\tilde u(t,\tilde \beta^*(t,x))~=~\tilde u(t,x).$$
\endproof

\vglue .5cm

\noindent {\bf Acknowledgements.} The work of R.M. Chen was partially supported by the Central Research Development Fund No. 04.13205.30205 from University of Pittsburgh. The work of Y. Liu is partially supported by  NSF grant DMS-1207840.

\bigskip

Conflict of Interest: The authors declare that they have no conflict of interest.

%

\end{document}